\documentclass[a4paper,11pt,reqno]{amsart}
\usepackage{amsmath}
\usepackage{amsfonts}
\usepackage{amssymb}
\usepackage{graphicx}
\usepackage{amsthm}
\usepackage{color}
\usepackage{url}
\usepackage{bigints}

\newtheorem{theorem}{Theorem}

\newtheorem{lemma}[theorem]{Lemma}

\newtheorem{proposition}[theorem]{Proposition}
\newtheorem{remark}[theorem]{Remark}

\numberwithin{theorem}{section} 
\numberwithin{equation}{section}

\newcommand{\meas}{\mathrm{meas}}

\makeatletter
\newcommand{\addresseshere}{%
  \enddoc@text\let\enddoc@text\relax
}
\makeatother

\providecommand{\norm}[1]{\lVert#1\rVert}

\usepackage[margin=1in]{geometry}

\begin{document}

\title[Sobolev-type inequalities on Cartan-Hadamard manifolds]{Sobolev-type inequalities on Cartan-Hadamard manifolds \\ and applications to some nonlinear diffusion equations}
\author{Matteo Muratori and Alberto Roncoroni}
\address{Matteo Muratori, Dipartimento di Matematica, Politecnico di Milano, Piazza Leonardo da Vinci 32, 20133 Milano, Italy}
\email{matteo.muratori@polimi.it}
\address{Alberto Roncoroni\footnote{Present address: Dipartimento di Matematica e Informatica “Ulisse Dini”, Universit\`a degli Studi di Firenze, Viale Morgagni 67/A, 50134 Firenze, Italy. Email address: \url{alberto.roncoroni@unifi.it}}, Dipartimento di Matematica ``Felice Casorati'', Universit\`a degli Studi di Pavia, Via Ferrata 5, 27100 Pavia, Italy}
\email{alberto.roncoroni01@universitadipavia.it}

\subjclass[2010]{Primary: 46E35. Secondary: 26D10, 35K10, 46E30, 58C40, 58J35.}
\keywords{Sobolev inequality; Poincar\'e inequality; Cartan-Hadamard manifolds; negative curvature; smoothing effect; porous medium equation.}

\begin{abstract}
We investigate the validity, as well as the failure, of Sobolev-type inequalities on Cartan-Hadamard manifolds under suitable bounds on the sectional and the Ricci curvatures. We prove that if the sectional curvatures are bounded from above by a negative power of the distance from a fixed pole (times a negative constant), then all the $ L^p $ inequalities that interpolate between Poincar\'e and Sobolev hold for \emph{radial} functions provided the power lies in the interval $ (-2,0) $. The Poincar\'e inequality was established by H.P.~McKean under a \emph{constant} negative bound from above on the sectional curvatures. If the power is equal to the critical value $ -2 $ we show that $ p $ must necessarily be bounded away from $ 2 $. Upon assuming that the Ricci curvature vanishes at infinity, the \emph{nonradial} version of such inequalities turns out to \emph{fail}, except in the Sobolev case. Finally, we discuss applications of the here-established Sobolev-type inequalities to \emph{optimal} smoothing effects for {radial} porous medium equations.
\end{abstract}

\maketitle

\section{Introduction} 

It is well known that on any \emph{Cartan-Hadamard} manifold $M$ of dimension $ N \ge 3 $, namely a complete, noncompact, simply connected Riemannian manifold with everywhere nonpositive sectional curvatures, the Euclidean \emph{Sobolev inequality}
\begin{equation}\label{eq-sob-euc}
\left\| f \right\|_{L^{2^\ast}\!(M)} \leq  C_N \left\| \nabla f \right\|_{L^2(M)} \quad \forall f \in C_{c}^1(M) \, , \qquad 2^\ast := \frac{2N}{N-2}
\end{equation}
holds, for some constant $C_N >0 $ depending only on $N$. This result can be proved in several ways: see for instance \cite[Lemma 8.1, Theorem 8.3]{Hebey} or \cite[Corollary 14.23, Remark 14.24]{Grig09}. On the other hand, if the sectional curvatures are bounded from above by a negative constant $ -k $, then in addition to \eqref{eq-sob-euc} also the \emph{Poincar\'e} inequality 
\begin{equation}\label{eq-poin-intro} 
\left\| f \right\|_{L^2(M)} \le \frac{2}{\sqrt{k} \left( N-1 \right) } \left\| \nabla f \right\|_{L^2(M)} \qquad \forall f \in C^1_c(M)  
\end{equation}
holds, or equivalently the infimum of the spectrum of (minus) the Laplace-Beltrami operator on $M$ is bounded from below by $ k(N-1)^2/4 $, i.e.~$ \Delta $ has a \emph{spectral gap}. This is a celebrated result due to H.P. McKean \cite{McKean}, which we discuss extensively in the appendix (Theorem \ref{mck-orig}). Note that the spectral bound is optimal as it is attained by the \emph{hyperbolic space} $ \mathbb{H}^N $ of curvature $ -k $.

Hence, if the sectional curvatures are bounded from above by $ k=0 $ then $M$ supports the Euclidean Sobolev inequality \eqref{eq-sob-euc}, and as soon as such bound becomes strictly negative $M$ also supports the Poincar\'e inequality \eqref{eq-poin-intro}. In other words, there is a ``jump'' of the $ L^p $ exponent in the left-hand side of the inequality $ \| f \|_p \le C \, \| \nabla f \|_2 $ with respect to curvature. The naive question that gave rise to the present paper is: what happens in between? That is, suppose the sectional curvatures of $M$ (which we denote by $ \mathrm{Sect}(x) $) satisfy a bound of the following type:
\begin{equation}\label{sezionale-intro} 
\mathrm{Sect}(x) \leq -  \mathsf{C} \, r^{-\beta} \qquad \forall x \in M \setminus B_{R} \, ,
\end{equation}
for some $ \beta \in (0,2] $ and $ \mathsf{C},R>0$, where $ r\equiv r(x):=\operatorname{dist}(x,o) $ denotes the geodesic distance of $x$ from a fixed point $o$ (the \emph{pole}). Then, what kind of inequalities of the form
\begin{equation}\label{Sobolev-type}
\left\| f \right\|_{L^{p}(M)} \leq  C_p \left\| \nabla f \right\|_{L^2(M)} 
\end{equation}
does $M$ support? The answers we find are nontrivial; one has to separate the \emph{sub-hyperbolic} range ($ \beta \in (0,2) $) and the \emph{quasi-Euclidean} case ($ \beta=2 $), borrowing a terminology from \cite{GMV}. 

\vspace{0.1cm}
Let us briefly describe the theorems we prove, which are stated precisely in Section \ref{sect:stat}. If $ \beta \in (0,2) $ we show that \eqref{Sobolev-type} holds for \emph{radial} functions for all $ p \in (2,2^\ast] $, with
\begin{equation}\label{bhv-cp}
C_p = \frac{C \, p^{\frac{2+\beta}{2(2-\beta)}}}{(p-2)^{\frac{\beta}{2-\beta}}} \, ,
\end{equation}
$C$ being another positive constant that depends only on $ N , \beta , \mathsf{C} , R$. The result is optimal with respect to $p$ (Theorem \ref{teo}). In the critical case $\beta=2$, inequality \eqref{Sobolev-type} (still for radial functions) starts to hold from a certain exponent $ \tilde{2} \in (2,2^\ast) $ that depends on $ N $ and $ \mathsf{C} $. This result is also optimal with respect to $p$, see Theorem \ref{teo2}. Finally, we prove in Theorem \ref{Controex teo} that out of the radial setting there is no hope to extend such inequalities: it is enough to assume that the Ricci curvature vanishes uniformly at infinity (regardless of the rate) to be able to construct a sequence of nonradial functions that make the constant $C_p$ in \eqref{Sobolev-type} blow up for every $ p<2^\ast $. The case $ \beta>2 $ is not interesting, as the sole inequality of the type of \eqref{Sobolev-type} that holds, even if restricted to radial functions, is the standard Sobolev one (we refer to Subsection \ref{ext}). 

The techniques of proof that we exploit take advantage of two main tools: \emph{one-dimensional weighted} functional inequalities and \emph{Laplacian-comparison} theorems, which are recalled in Subsections \ref{one-weigh} and \ref{lc}, respectively. The idea is to first study the radial inequalities on \emph{model manifolds}, namely spherically-symmetric Riemannian manifolds whose metric $g$ reads
\begin{equation}\label{model metric-intro}
g  \equiv dr^2 + \psi(r)^2 \, d\theta_{\mathbb{S}^{N-1}}^2 
\end{equation}
for some regular ``model'' function $\psi: \mathbb{R}^+ \rightarrow \mathbb{R}^+ $ (see Definition \ref{def-A}), where $ d\theta_{\mathbb{S}^{N-1}}^2$ is the usual metric on the $ (N-1) $-dimensional unit sphere $ \mathbb{S}^{N-1} $. In this framework, \eqref{Sobolev-type} becomes a family of one-dimensional weighted inequalities, where the associated weight is just $ \psi(r)^{N-1} $. Then, by carefully invoking Laplacian and \emph{surface-measure comparison} (see also Subsection \ref{nfrg}), one can apply the same arguments to general Cartan-Hadamard manifolds. 

\vspace{0.1cm}
For simplicity, above we have assumed $ N \ge 3 $ and that \eqref{sezionale-intro} holds for sectional curvatures. Actually, our results are also valid in the $ 2 $-dimensional case (up to excluding $ p=2^\ast=\infty $), and in fact one can replace $ \mathrm{Sect}(x) $ with $ \mathrm{Sect}_\omega(x) $, the latter symbol denoting ``radial'' sectional curvatures (we refer to Subsection \ref{bas-not}). In particular, we can give an alternative proof of McKean's Theorem under slightly weaker curvature assumptions (Theorem \ref{thm:mckean}).

\vspace{0.1cm}
A motivation to study Sobolev-type inequalities on Cartan-Hadamard manifolds under curvature bounds like \eqref{sezionale-intro} came from \cite{GMV}, where the authors investigate the asymptotic behavior of nonnegative solutions to the \emph{porous medium equation} (see the monograph \cite{V2})
\begin{equation}\label{PME}
\begin{cases}
u_t = \Delta (u^m) & \text{in } M \times \mathbb{R}^+ \, , \\ 
u=u_0 \ge 0 & \text{on } M \times \lbrace 0 \rbrace \, ,
\end{cases}
\end{equation}
where $ m>1 $. They prove that if $ \mathrm{Sect}(x) \approx r^{-\beta} $ for some $ \beta \in (0,2) $ then 
\begin{equation*}\label{asym-pme}
u(x,t) \approx t^{-\frac{1}{m-1}} \left[ \gamma \left( \log t \right)^{\frac{2+\beta}{2-\beta}} - r^{\frac{2+\beta}{2}} \right]_+^{\frac{1}{m-1}}  \qquad \text{as } t \to \infty 
\end{equation*}
for a suitable positive constant $ \gamma $, provided $ u_0 $ is compactly supported. Such bounds are proved by pure barrier methods, but they are compatible with the $ L^1 $-$ L^\infty $ \emph{smoothing effects} shown in \cite{GM16} upon \emph{assuming} that \eqref{Sobolev-type} holds with a constant $C_p$ as in \eqref{bhv-cp}. By combining these results, we conjectured the validity of such (radial) inequalities. Analogous connections can be established in the quasi-Euclidean case. For more details, we refer the reader to Section~\ref{PME-application}. 

\vspace{0.1cm}

\emph{Sobolev, Poincar\'e and Hardy-type inequalities on manifolds}. The investigation of functional inequalities on Riemannian manifolds is a very wide and active research field: it is out of sight to give a complete list of all the related literature. We limit ourselves to mentioning articles or monographs which, either by dealing with topics of a wider scope or more specific ones, are connected with the main achievements of the present paper. From a general point of view, we quote \cite{Hebey bis,Hebey}, where the discussion is devoted to the rigorous definition of \emph{Sobolev spaces} on Riemannian manifolds and the properties of the associated embeddings.

The very first result dealing with the \emph{optimal constant} of the Euclidean Sobolev inequality is due to the celebrated papers \cite{Aubin2} and \cite{Talenti}, by T.~Aubin and G.~Talenti, respectively. Then the first author continued the investigation of Sobolev-type inequalities as well as related optimality issues on Riemannian manifolds: see \cite{Aubin1} (compact manifolds with applications to the \emph{Yamabe problem}), \cite{Aubin3} (where higher-order inequalities are also considered) and \cite{Aubin Li} (estimates of the best constants of subcritical Sobolev embeddings). Some of the results of \cite{Aubin2} were later improved in \cite{HV,Hebey ter}, essentially by requiring bounds on the Ricci curvature in place of the sectional curvatures. Note that, in the light of the recent breakthrough paper \cite{GS}, we can now assert that the optimal Sobolev constant in \eqref{eq-sob-euc} on any Cartan-Hadamard manifold is in fact Euclidean (see also \cite[Section 8]{Hebey}). In any case, let us point out that here we do not investigate the exact value of optimal constants: we are only interested in the dependence of the latter w.r.t.~$p$.

In contrast to McKean's Theorem, in \cite{LZ} it was shown that, on any complete noncompact Riemannian manifold, the essential spectrum of the Laplace-Beltrami operator starts from zero as soon as the Ricci curvature vanishes at infinity. This generalizes \cite{W}, where the same thesis was established upon assuming an at-least-quadratic decay of the negative curvatures.

Functional-analytic issues on the hyperbolic space $ \mathbb{H}^N $ have drawn a lot of interest recently, the latter being in a sense the simplest example of a noncompact Riemannian manifold with negative curvatures. In this regard, we mention \cite{BGG}, where an improved version of the Poincar\'e inequality is obtained with optimal \emph{remainder terms} of Hardy type.
In \cite{Car1} the author also considered Hardy-type inequalities, for nonstandard weights satisfying suitable differential equations (with applications to Cartan-Hadamard manifolds). Finally, in \cite{Nguyen} it is established an inequality on $ \mathbb{H}^N $ that yields the optimal Sobolev and Poincar\'e constants simultaneously. 

For a survey dealing with connections between the Poincar\'e inequality, the logarithmic Sobolev inequality, measure-concentration issues and isoperimetric bounds, we refer to \cite{L04}. In \cite{BLW} the authors, starting from inequalities that interpolate between Poincar\'e and log-Sobolev, provide a method to obtain suitable \emph{weighted} inequalities of the same type. 

Most of the results we have mentioned above are focused on proving the validity of Sobolev (or Poincar\'e, Hardy) inequalities. One can also investigate \emph{topological consequences}: in \cite{L99,Car2} it is shown (under suitable curvature or volume-growth assumptions, respectively) that a Riemannian manifold supporting a Sobolev inequality with Euclidean constant is necessarily isometric to $\mathbb{R}^N$. For similar problems, but rather different methods, see also \cite{PiVe}. 

\vspace{0.1cm}
\emph{Functional inequalities and nonlinear diffusions}. The link between the Sobolev inequality (or Gagliardo-Nirenberg and Nash inequalities in low dimensions) and a sharp decay estimate for the heat kernel is a well-studied topic, which goes back to some pioneering works between the 50s and the 80s: see the monograph \cite{Dav}. In this regard, let us also mention \cite{GSC}, where Faber-Krahn inequalities on Riemannian manifolds under removal of compact subsets or gluing of noncompact manifolds are investigated, with applications to heat-kernel bounds. 

In the last decades, several results that connect the validity of functional inequalities of Sobolev, log-Sobolev or Poincar\'e type with smoothing effects for \emph{nonlinear} diffusion equations have been established: see \cite{GMP13} for weighted porous medium equations in the presence of weights and Poincar\'e inequalities, \cite{GM13,FM} for optimal short and long-time smoothing estimates for porous medium equations (or the more general \emph{filtration equation}) on Euclidean domains in the case of homogeneous Neumann conditions, and \cite{GM16} for similar analyses focused on the consequences of the validity of families of Sobolev-type inequalities of the type of \eqref{Sobolev-type}. Previous results in this direction can be found in \cite{BG05}. As a general reference on smoothing effects, we also quote \cite{V1}. It is worth pointing out that in order to prove some of the main theorems of these papers, the authors often exploit a remarkable equivalence tool between \emph{families} of Gagliardo-Nirenberg-Sobolev inequalities and a \emph{single} inequality, which is due to \cite{BCLS}. 

Among recent works that take advantage of connections between functional inequalities and \emph{fast diffusion} flows (i.e.~\eqref{PME} with $m<1$) we refer to \cite{DEL,DELM}, where the latter have been thoroughly exploited in order to prove or disprove the achievement of optimality by radial functions in Caffarelli-Kohn-Nirenberg inequalities. For a functional-analytic investigation of fast diffusions on Cartan-Hadamard manifolds, see also \cite{BGV08}. Finally, the monograph \cite{BGL} is devoted to a wide-scope discussion on the interplay between analytic, geometric and probabilistic features of Markov diffusion semigroups, which involves functional inequalities related to those treated here and curvature-dimension conditions in more general metric frameworks.

\subsection{Plan of the paper}

In Section \ref{sect:stat} we state and discuss our main results, after a brief introduction to notations. Section \ref{sect2} recalls some preliminary tools in Riemannian geometry and functional inequalities. In Section \ref{sec:radial} we carry out the proofs in the radial setting (Theorems \ref{teo} and \ref{teo2}). Section \ref{sect:nonrad} deals with the nonradial case, namely the disproof of the Sobolev-type inequalities for general nonradial functions (see Theorem \ref{Controex teo}). The last section is devoted to illustrate how the results established here entail optimal smoothing estimates for the (radial) porous medium equation flow on Cartan-Hadamard model manifolds (Theorems \ref{thm: sobo-deg} and \ref{thm: sobo-deg-quasi}). In Appendix \ref{sec-mckean} we focus on the Poincar\'e case $ p=2 $ and give an alternative proof of McKean's inequality (see Theorems \ref{thm:mckean} and \ref{mckball}). 

\section{Statements of the main results}\label{sect:stat} 

In this section, first we introduce some basic notations and definitions that will be used below, then we provide the statements of our main results (whose proofs are deferred to Sections \ref{sec:radial}--\ref{sect:nonrad})  and discuss their optimality or possible related extensions. 

\subsection{Basic notations}\label{bas-not}

Given an $ N $-dimensional Riemannian manifold $(M,g)$ and $x\in M$, we let $\mathrm{Sect}(x)$ denote the sectional curvature w.r.t.~\emph{any} $2$-dimensional tangent subspace at $ x \in {M} $ and $\mathrm{Ric}(x)$ the Ricci curvature at $ x \in {M} $, as a quadratic form on the tangent space $ T_xM $.

\vspace{0.1cm}
The symbol by $C^{1}_{c :\mathrm{rad}}(M)$ will stand for the space of all $C^1$ functions on $M$, with compact support, that are radial with respect to some (fixed) point $ o \in M $, i.e. 
\begin{equation*}
C^1_{c:\mathrm{rad}}(M) := \left\{ f \in C^1_{c}(M) : \ f(x) \equiv f( \operatorname{d}(x,o) ) \quad \forall x \in M \right\} ,
\end{equation*}
where $ r\equiv r(x):=\operatorname{d}(x,o)$ is the geodesic distance from $x$ to $o$. Furthermore, we let $d\nu$ denote the Riemannian volume measure of $M$ and set
\begin{equation*}
\norm{f}_{L^p(M)}:=\left(\int_M |f|^p\, d\nu\right)^{\frac{1}{p}}
\end{equation*}
for all $1\leq p < \infty$, while $ \norm{f}_{L^\infty(M)} $ as usual represents the essential supremum of $ |f| $. The \emph{Sobolev critical exponent} is defined in \eqref{eq-sob-euc} for $ N \ge 3 $, while for $ N=2 $ we put $ 2^\ast =\infty $. 

\vspace{0.1cm}
In many parts of the paper it will be necessary to deal with Riemannian manifolds whose metric has the special structure \eqref{model metric-intro}, where $\psi: \mathbb{R}^+ \rightarrow \mathbb{R}^+ $ is a function belonging to the class
\begin{equation}\label{def-A} 
\mathcal{A} := \left\{ \psi\in C^{\infty}((0,\infty))\cap C^{1}([0,\infty)) : \ \psi(0)=0 \, , \ \psi(r)>0 \ \, \forall r>0 \, , \ \psi^\prime(0)=1 \right\} .
\end{equation}
These objects are referred to as \emph{model manifolds}, and will be denoted by $ \mathbb{M}_\psi^N $. The Euclidean space corresponds to $ \psi(r)=r $, while the hyperbolic space corresponds to $ \psi(r)= \sinh r $. 

\vspace{0.1cm}
By an $ N $-dimensional \emph{Cartan-Hadamard} manifold $ M \equiv \mathbb{M}^N $ we mean a complete, noncompact, simply connected Riemannian manifold with everywhere nonpositive sectional curvatures. On such manifolds, the cut locus of any point $o$ is empty; hence, for every $ x \in M \setminus \lbrace o \rbrace$ we can define polar coordinates with pole at $o$, namely $r = \operatorname{d}(x,o)$ and $\theta\in\mathbb{S}^{N-1}$. We then let $B_{r}$ denote the geodesic ball of radius $r$ centered at $o$ and set $S_{r}:=\partial B_{r}$. The symbol $ \textrm{Sect}_\omega(x) $ will stand for the sectional curvature w.r.t.~to any 2-dimensional tangent subspace $ \omega $ at $x$ containing the \emph{radial} direction, and $ \mathrm{Ric}_o(x) $ for the Ricci curvature at $x$ evaluated in the radial direction. 

\subsection{The results}\label{sect:stat-1} 

If the (radial) sectional curvatures of a Cartan-Hadamard manifold do not decay too fast at infinity, i.e.~slower than a negative quadratic power of $r$, we can show that all the radial inequalities from Poincar\'e (excluded) to Sobolev hold. 

\begin{theorem}\label{teo}
Let $ \mathbb{M}^N $ be a Cartan-Hadamard manifold such that 
\begin{equation}\label{sezionale} 
\mathrm{Sect}_\omega(x) \leq - C_0 \, r^{-\beta} \qquad \forall x \in \mathbb{M}^N \setminus B_{R_0} \, ,
\end{equation}
for some $ \beta \in (0,2) $ and $C_0,R_0>0$.
Then there exists a positive constant $C$, depending only on $ N , \beta , C_0 , R_0$, such that for every $ p \in\left(2,2^\ast\right] \cap (2,\infty) $ the {radial} Sobolev-type inequality
\begin{equation}\label{tsfinale}
\left\| f \right\|_{L^p\left(\mathbb{M}^N\right)} \leq  \frac{C \, p^{\frac{2+\beta}{2(2-\beta)}}}{(p-2)^{\frac{\beta}{2-\beta}}} \left\| \nabla f \right\|_{L^2\left(\mathbb{M}^N\right)}\qquad \forall f \in C_{c:\mathrm{rad}}^1\!\left(\mathbb{M}^N\right) 
\end{equation}
holds. Moreover, the dependence on $p$ of the multiplying constant in \eqref{tsfinale} is optimal, in the sense that for each $ \beta \in (0,2) $ there exists a model manifold $ \mathbb{M}^N_\psi $, complying with \eqref{sezionale}, such that  
\begin{equation}\label{tsfinale-opt}
\inf_{f \in C_{c:\mathrm{rad}}^1\!\left(\mathbb{M}^N_\psi\right), \, f \not\equiv 0 } \frac{\left\| \nabla f \right\|_{L^2\left(\mathbb{M}^N_\psi \right)}}{\left\| f \right\|_{L^p\left(\mathbb{M}^N_\psi \right)}} \le \frac{(p-2)^{\frac{\beta}{2-\beta}}}{C \, p^{\frac{2+\beta}{2(2-\beta)}}} \qquad \forall p \in (2,2^\ast] \cap (2,\infty)
\end{equation} 
for another positive constant $ C $ depending on $ N , \beta , C_0 , R_0$.
\end{theorem}

If the curvatures decay with a rate which is {at most} quadratic, we still have radial Sobolev-type inequalities; however, they start to hold from a certain exponent {strictly larger than} $2$. 
\begin{theorem}\label{teo2}
Let $ \mathbb{M}^N $ be a Cartan-Hadamard manifold such that 
\begin{equation}\label{sezionale2} 
\mathrm{Sect}_\omega(x) \leq - C_1 \, r^{-2} \qquad \forall x \in \mathbb{M}^N \setminus B_{R_0} \, ,
\end{equation}
for some $ C_1 ,R_0>0$.
Then there exists a positive constant $C$, depending only on $ N, C_1 , R_0$, such that for every $ p \in \left[\tilde{2},2^\ast\right] \cap \left[ \tilde{2} , \infty \right) $ the {radial} Sobolev-type inequality
\begin{equation}\label{tsfinale2}
\left\| f \right\|_{L^p\left(\mathbb{M}^N\right)} \leq  C \, \sqrt{p} \left\| \nabla f \right\|_{L^2\left(\mathbb{M}^N\right)} \qquad \forall f \in C_{c:\mathrm{rad}}^1\!\left(\mathbb{M}^N\right) 
\end{equation}
holds, where 
\begin{equation}\label{tsfinaleX}
\tilde{2} := \frac{2\tilde{N}}{\tilde{N}-2} \, , \qquad \tilde{N} := \frac{N+1+\sqrt{1+4C_1}\,(N-1)}{2} \, .
\end{equation}
Moreover, the result is optimal w.r.t.~to $p$, in the sense that for each $ C_1>0 $ there exists a model manifold $ \mathbb{M}^N_\psi $, complying with \eqref{sezionale2}, such that \eqref{tsfinale2} fails for all $ p < \tilde{2}  $ and, in the case $ N=2 $, the best constant in \eqref{tsfinale2} behaves like $ \sqrt{p} $ (up to multiplicative constants) as $ p \to \infty $. 
\end{theorem}

Finally, we establish the failure of the above results for general nonradial functions, as soon as the Ricci curvature vanishes uniformly at infinity.
\begin{theorem}\label{Controex teo}
Let $\mathbb{M}^N$ be a Cartan-Hadamard manifold such that 
\begin{equation}\label{ricci-chch tesi}
\lim_{r \to \infty} \mathrm{Ric}(x) = 0  \, .
\end{equation}
Let $ p \in [2, 2^\ast) $. Then there exists no positive constant $C$ such that $ \mathbb{M}^N $ supports the Sobolev-type inequality
\begin{equation}\label{sob-one-nnrad proof}
\left\| f \right\|_{L^p\left(\mathbb{M}^N\right)} \leq  C \left\| \nabla f \right\|_{L^2\left(\mathbb{M}^N\right)}  \qquad \forall f \in C_c^1( \mathbb{M}^N ) \, .
\end{equation}
\end{theorem}

Note that \eqref{ricci-chch tesi}, in the special Cartan-Hadamard framework, coincides with the $ \liminf $ condition required in \cite[Theorem 1]{LZ}, where the case $ p=2 $ was treated on general manifolds.

\subsection{Optimality and possible extensions}\label{ext}

In the sequel we collect a series of observations regarding the assumptions, the (possibly wider) scope and the consequences of our results.

\vspace{0.1cm}

\noindent\emph{Ricci bounds from above}. The thesis of Theorem \ref{teo} still holds if one replaces $ \mathrm{Sect}_\omega(x) $ with $ \mathrm{Ric}_o(x) $ in \eqref{sezionale}: this is due to the fact that Laplacian-comparison  with model manifolds, which we exploit extensively throughout the paper, can also be established under such a weaker hypothesis. The argument applies to Theorem \ref{teo2} as well, except that in this case the analogue of the exponent $ \tilde{2} $ in \eqref{tsfinaleX} is no more optimal, since one has \eqref{lap-weak} in place of \eqref{comp-sect-2}. The same holds for Theorems \ref{thm:mckean} (no optimal constant however) and \ref{mckball} in the appendix.
	
\vspace{0.1cm}	

\noindent\emph{Multiplicative constants}. We stress that the constants $C$ appearing in the Sobolev-type inequalities \eqref{tsfinale} and \eqref{tsfinale2} do not depend on the manifold $ \mathbb{M}^N $, but only on the spatial dimension and the curvature bounds. This suggests that perhaps analogous inequalities may hold on rougher geometric frameworks (e.g.~allowing the curvature to be $ -\infty $ somewhere). 

\vspace{0.1cm}	

\noindent\emph{Radiality and optimality of the powers}.
	Theorem \ref{Controex teo} implies that the conclusions of Theorems \ref{teo} and \ref{teo2} {cannot} hold, in general, for nonradial functions: indeed, for each $ \beta \in (0,2] $, it is enough to consider a Cartan-Hadamard manifold (e.g.~a model) satisfying
	\begin{equation}\label{sect-asymp}
	\mathrm{Sect}(x) \sim r^{-\beta} \qquad \text{as } r \to \infty \, .
	\end{equation}
	For similar reasons, we do not treat the case $ \beta>2 $: by following a strategy that goes along the lines of the proof of optimality in Theorem \ref{teo2}, it is not difficult to check that for each $ \beta > 2 $ one can construct a model manifold complying with \eqref{sect-asymp} for which all the radial inequalities in \eqref{tsfinale} fail if $ p<2^\ast $ (we omit details and refer to \cite[Subsection 2.3, Type IV]{GMV}). Hence, no inequality of the type of \eqref{tsfinale} other than the classical Sobolev one can be valid.

\vspace{0.1cm}	

\noindent\emph{Extension  to nonsmooth functions}.
	For simplicity, we have stated Theorems \ref{teo} and \ref{teo2} (as well as Theorems \ref{mck-orig}, \ref{thm:mckean} and \ref{mckball} below) for functions in $ C^1_{c}(\mathbb{M}^N) $. Nevertheless, by means of standard density arguments, it is plain that they also hold for compactly-supported {Lipschitz} functions or, more in general, for all functions belonging to the closure of $ C^1_{c}(\mathbb{M}^N) $ (or its radial counterpart) w.r.t.~the $ L^2 $ norm of the gradient.

\vspace{0.1cm}	

\noindent\emph{On the Cartan-Hadamard assumption}.
	It is worth pointing out that, in general, we cannot drop the assumption that $ M $ is Cartan-Hadamard, i.e.~it is not enough to merely require that $ M $ has a pole and the sectional curvatures satisfy \eqref{sezionale} or \eqref{sezionale2}. Indeed, let us consider a model manifold $ \mathbb{M}^N_\psi $ with $\psi(r) = e^{-r^\alpha}$ for large $r$, where $\alpha = \beta/2 \in (0,1) $. It is straightforward to check that such a manifold complies with \eqref{sezionale}; however, it is apparent that all the inequalities in \eqref{tsfinale} fail, since the supremum appearing in \eqref{sob-one-1}  is identically $ \infty $ because of the second integral. Similarly, regarding \eqref{sezionale2}, one can consider a model manifold $ \mathbb{M}^N_\psi $ with $ \psi(r) = r^{-q_2}$ for large $r$, where $ q_2 > 0 $ is the same power as in \eqref{ee55}. 

\vspace{0.1cm}	

\noindent\emph{The case $ p \in [1,2) $}.
	From the very beginning we have assumed that $ p \ge 2 $. In fact there is a simple reason for such a restriction: it was proved in \cite[Theorem 4.6]{GM16} that on {any} Cartan-Hadamard manifold inequality \eqref{Sobolev-type} necessarily fails as soon as $ p $ is strictly smaller than $2$. Moreover, since the argument used in the corresponding proof only relies on radial functions, the inequality is false even if restricted to $ C^1_{c:\mathrm{rad}}(\mathbb{M}^N) $.
	
\section{Geometric and functional preliminaries}\label{sect2} 

In the following, we recall some basic facts in Riemannian geometry concerning volume and surface measure and Laplacian comparison (Subsections \ref{nfrg}--\ref{lc}), along with a key result related to weighted one-dimensional Sobolev-type inequalities (Subsection \ref{one-weigh}).

\subsection{Volume, surface and Laplacian}\label{nfrg} 
We adopt the same notations as in Subsection \ref{bas-not}. If $ \mathbb{M}^N $ is a Cartan-Hadamard manifold, then the surface measure of geodesic spheres reads
\begin{equation}\label{def-meas}
\meas(S_{r})=\int_{\mathbb{S}^{N-1}}A(r,\theta)\, d\theta \, , \qquad \text{where } \, d\theta := d \theta_1 \ldots d\theta_{N-1} 
\end{equation}
and $ A(r,\theta) $ is the weight associated with the volume measure of $ \mathbb{M}^N $ w.r.t.~polar coordinates, which turns out to be the square root of the determinant of the metric matrix written in such coordinates (see e.g.~\cite[Section 3]{Grigor'yan}). In particular, the latter is a nonnegative $ C^\infty(\mathbb{R}^+ \times \mathbb{S}^{N-1}) $ function. 
Hence, if $ f \in L^p(M) $ we have 
\begin{equation}\label{LP}
\begin{aligned}
\norm{f}_{L^p(M)} =\left(\int_M |f|^p\, d\nu\right)^{\frac{1}{p}} = & \left(\int_0^\infty\int_{\mathbb{S}^{N-1}}|f(r,\theta)|^p \, A(r,\theta) \, d\theta \, dr\right)^{\frac{1}{p}} \\
& \! \! \! \! \! \! \! \! \! \! \! \! \! \! \! \! \overset{\text{if $f$ is radial}}{=} \left(\int_0^\infty |f(r)|^p \, \meas(S_r) \, dr\right)^{\frac{1}{p}} ,
\end{aligned}
\end{equation}
so that radial Sobolev-type inequalities can be rewritten as one-dimensional weighted inequalities (see Subsection \ref{one-weigh}). The Laplace-Beltrami (or Laplacian) operator on $ \mathbb{M}^N $ reads
\begin{equation*}
\Delta=\dfrac{\partial^2}{\partial r^2} + \mathsf{m}(r,\theta) \, \dfrac{\partial}{\partial r}+\Delta_{S_{r}} \, ,
\end{equation*}
where $\Delta_{S_{r}}$ is the Laplace-Beltrami operator on the submanifold $S_{r}$ and 
\begin{equation}\label{lap-m}
\mathsf{m}(r,\theta) = \dfrac{\partial}{\partial r} \left( \log A(r,\theta) \right) \qquad \forall x \equiv (r,\theta) \in \mathbb{R}^+ \times \left( \mathbb{S}^{N-1} \setminus \mathcal{P} \right) ,
\end{equation} 
where $ \mathcal{P} $ is the (negligible) set of angles $ \theta \equiv (\theta_1,\ldots,\theta_{N-1}) \in \mathbb{S}^{N-1} $ at which $ A(r,\theta) $ vanishes identically. Note that $ \mathsf{m}(r,\theta) $ is precisely the \emph{Laplacian of the distance function} $ x \equiv (r,\theta) \mapsto r$. By integrating \eqref{lap-m} from a fixed $r_0>0$ to $ r > r_0 $ we deduce that 
\begin{equation*}\label{eq:int-m}
\int_{r_{0}}^r \mathsf{m}(s,\theta)\, ds=\log A(r,\theta) - \log A(r_0,\theta) \qquad \Longrightarrow \qquad  A(r,\theta)=e^{\int_{r_0}^r \mathsf{m}(s,\theta)\, ds +c_{\theta}}  \, ,
\end{equation*}
with $c_{\theta} := \log A(r_0,\theta) $. As a result, recalling \eqref{def-meas}, we can write
\begin{equation}\label{meas}
\meas(S_{r})=\int_{\mathbb{S}^{N-1}}e^{\int_{r_0}^r \mathsf{m}(s,\theta)\, ds+c_{\theta}}\, d\theta \, .
\end{equation}

\subsection{Laplacian-comparison theorems}\label{lc} 
Classical results allow one to compare the Laplacian (and the Hessian in some cases) of the distance function of a Cartan-Hadamard manifold with the Laplacian of the distance function of a suitable model manifold achieving equality in the curvature bounds (as a reference see e.g.~\cite[Section 2]{GreeneWu} or \cite[Section 15]{Grigor'yan}). More precisely, if 
\begin{equation}\label{comp-sect}
\mathrm{Sect}_\omega(x)\leq - \dfrac{\psi''(r)}{\psi(r)} \qquad  \forall x \equiv(r,\theta) \in M \setminus\lbrace o\rbrace
\end{equation}
for some function $\psi\in\mathcal{A}$, then
\begin{equation}\label{comp-sect-2}
\mathsf{m}(r,\theta) \geq (N-1) \, \dfrac{\psi'(r)}{\psi(r)} \qquad \forall (r,\theta) \in \mathbb{R}^+ \times \mathbb{S}^{N-1} \, .
\end{equation}
Similarly, if 
\begin{equation*}
\mathrm{Ric}_o(x) \geq- (N-1) \, \dfrac{\psi''(r)}{\psi(r)} \qquad \forall x \equiv (r,\theta) \in M \setminus \lbrace o \rbrace 
\end{equation*}
for another function $\psi\in\mathcal{A}$, then
\begin{equation*}
\mathsf{m}(r,\theta)\leq (N-1) \,\dfrac{\psi'(r)}{\psi(r)} \qquad \forall (r,\theta) \in \mathbb{R}^+ \times \mathbb{S}^{N-1} \, .
\end{equation*}
Although we will only use the above inequalities in the Cartan-Hadamard setting, we point out that they are true on more general Riemannian manifolds (at least manifolds with a pole). As a simple consequence of Laplacian comparison with the Euclidean space (i.e.~\eqref{comp-sect}--\eqref{comp-sect-2} with $ \psi(r)=r $), on any Cartan-Hadamard manifold we have
\begin{equation}\label{lap-euc}
\mathsf{m}(r,\theta) \ge \frac{N-1}{r} \qquad \forall (r,\theta) \in \mathbb{R}^+ \times \mathbb{S}^{N-1} \, .
\end{equation} 
Actually, a comparison result similar to the first one can be deduced by replacing the (radial) sectional curvatures $ \mathrm{Sect}_\omega(x) $ with the radial Ricci curvature. Namely, if 
\begin{equation*}
\mathrm{Ric}_o(x)\leq - (N-1) \, \dfrac{\psi''(r)}{\psi(r)} \qquad \forall x \equiv(r,\theta) \in \mathbb{M}^N \setminus \lbrace o\rbrace 
\end{equation*}
for some function $\psi\in\mathcal{A}$, then
\begin{equation}\label{lap-weak}
\mathsf{m}(r,\theta) \geq \sqrt{N-1} \, \dfrac{\psi'\!\left(\sqrt{N-1} \, r\right)}{\psi\!\left(\sqrt{N-1} \, r\right)} \qquad \forall(r,\theta) \in \mathbb{R}^+ \times \mathbb{S}^{N-1} \, .
\end{equation}
This is basically due to the fact that the \emph{Hessian} of the distance function on $ \mathbb{M}^N $ on a Cartan-Hadamard manifold has nonnegative eigenvalues: we refer to \cite[Theorem 2.15]{X}.

By exploiting Laplacian comparison with model functions $ \psi \in \mathcal{A} $ carefully chosen, one can prove the following (for the details see e.g.~\cite[Lemma 4.1]{GMV}). 
\begin{lemma}\label{lem-GMV}
Let $ \mathbb{M}^N $ be a Cartan-Hadamard manifold satisfying \eqref{sezionale} for some $ \beta \in [0,2) $ and $C_0,R_0>0$. Then there exist $ r_0=r_0(\beta,C_0,R_0)>0 $ and $ c=c(N,\beta,C_0,R_0)>0 $ such that
\begin{equation*}\label{m-1}
\mathsf{m}(r,\theta) \ge c \, r^{-\frac{\beta}{2}} \qquad \forall (r,\theta) \in [r_0,\infty) \times \mathbb{S}^{N-1} \, .
\end{equation*} 
\end{lemma} 

\subsection{One-dimensional weighted inequalities}\label{one-weigh}
In the following, by a \emph{weight} in $ \mathbb{R}^+ $ we simply mean any positive $ L^1_{\mathrm{loc}}([0,\infty)) $ function, even though in the rest of the paper we will in fact deal with more regular functions. The techniques we exploit in Section \ref{sec:radial}  take advantage of some known properties for \emph{one-dimensional} weighted Sobolev-type inequalities, or Hardy-type inequalities according to the terminology of the monograph \cite{Kufner}. We mainly refer to the latter, which collects several results in this direction (not only in the one-dimensional framework). 

\begin{proposition}[{\cite[Theorem 6.2]{Kufner}}]\label{thm:ko}
Let $w$ be a weight in $ \mathbb{R}^+ $. Let $ p \in [2,\infty) $. Then the Sobolev-type inequality
\begin{equation}\label{sob-one}
\left( \int_0^\infty \left| g(r) \right|^p  w(r) \, dr \right)^{\frac 1p} \le C \left( \int_0^\infty \left| g^\prime(r) \right|^2 w(r) \, dr \right)^{\frac 12}  \qquad  \forall g \in C_c^1([0,\infty))
\end{equation}
holds for some $ C>0 $ if and only if 
\begin{equation}\label{sob-one-1}
\mathcal{B}(w,p):= \sup_{r\in(0,\infty)} \left(\int_0^r w(s) \, ds \right)^{\frac{1}{p}} \left(\int_r^\infty \frac{1}{w(s)} \, ds \right)^{\frac 1 2} < \infty \, ,
\end{equation} 
and the best constant $C$ appearing in \eqref{sob-one} satisfies the two-sided bound
\begin{equation}\label{sob-one-2}
\mathcal{B}(w,p) \le C \le \left( 1+\tfrac p2 \right)^{\frac 1p} \left(  1+\tfrac 2p \right)^{\frac 12}  \, \mathcal{B}(w,p)  \, .
\end{equation}
\end{proposition}

\section{Proofs of the results for radial functions}\label{sec:radial} 
We devote this section to the proof of Theorems \ref{teo} and \ref{teo2}, so that the focus here is entirely on radial functions. Nonradial issues will be addressed in Section \ref{sect:nonrad} and in Appendix \ref{sec-mckean}.
\subsection{The sub-hyperbolic range}\label{sec:sub-hyp}

We start by showing that, under suitable hypotheses on the weight $ w(r) = \psi(r)^{N-1} $, which correspond to choosing $ \psi \in \mathcal{A} $ in the ``sub-hyperbolic'' range, the supremum appearing in \eqref{sob-one-1} can be bounded quantitatively.

\begin{lemma}\label{lem-psi}
Let $ \psi \in \mathcal{A} $ satisfy the following assumptions:
\begin{equation}\label{eq:curv}
\psi(r) \ge r \quad \forall r \ge 0 \, , \qquad \psi^\prime(r) \ge 0 \quad \forall r \ge 0 \, , \qquad \frac{\psi^\prime(r)}{\psi(r)} \geq \dfrac{c}{r^{\alpha}} \quad \forall r \ge r_0 \, ,
\end{equation} 
for some $ \alpha \in (0,1) $ and $ c,r_0>0 $. Let $ N \in \mathbb{N} $ with $ N \ge 2 $. Then there exists a positive constant $C$, depending only on $ N , \alpha , c , r_0$, such that for every $ p \in\left(2,2^\ast\right) $ we have
\begin{equation}\label{cost} 
\sup_{r\in(0,\infty)} \left(\int_0^r\psi(s)^{N-1} \, ds\right)^{\frac{1}{p}} \left(\int_r^\infty \frac{1}{\psi(s)^{N-1}} \, ds \right)^{\frac 1 2}\leq \frac{C \, p^{\frac{1+\alpha}{2(1-\alpha)}}}{(p-2)^{\frac{\alpha}{1-\alpha}}} \, . 
\end{equation}
\end{lemma}
\begin{proof}
First of all, let us establish that the l.h.s.~of \eqref{cost} is finite. To this end, set
\begin{equation}\label{eq:def-Q}
Q(r) := \left(\int_0^r\psi(s)^{N-1} \, ds\right)^{\frac{1}{p}}\left(\int_r^\infty \frac{1}{\psi(s)^{N-1}} \, ds \right)^{\frac 1 2} \qquad \forall r > 0 \, .
\end{equation}
The first and last inequality in \eqref{eq:curv} easily yield the bound from below
\begin{equation}\label{eq:stima-psi}
\psi(r) \ge \kappa \, e^{\frac{c}{1-\alpha} r^{1-\alpha}} \quad \forall r \ge r_0 \, , \qquad \text{where } \, \kappa \equiv \kappa(\alpha,c,r_0):={r_0} \, e^{-\frac{c}{1-\alpha} r_0^{1-\alpha}} ,
\end{equation}
which in particular ensures that $ Q(r) $ is smooth on $ (0,\infty) $. In addition, because $ \psi(r) \sim r $ as $ r \to 0 $ and $ p < 2^\ast $, it follows that $ \lim_{r \to 0} Q(r) = 0 $. In order to deal with the behavior of $ Q(r) $ at infinity, we need more estimates. To this aim, let us rewrite the last inequality in \eqref{eq:curv} as 
\begin{equation}\label{disfond} 
\psi(r)^{N-1} \le \frac{r^\alpha}{c(N-1)} \, \dfrac{d }{dr} \left( \psi^{N-1} \right)\!(r)  \qquad \forall r \ge r_0 \, .
\end{equation} 
Using \eqref{disfond} and integrating by parts between $ r_0 $ and $ r>r_0 $, we obtain: 
\begin{equation}\label{est-psi-parts}
c(N-1) \int_{r_0}^r \psi(s)^{N-1} \, ds  \le  r^{\alpha} \, \psi(r)^{N-1} - r_{0}^{\alpha} \, \psi(r_{0})^{N-1} - \alpha \int_{r_0}^r \frac{\psi(s)^{N-1} }{s^{1-\alpha}} \, ds \leq r^{\alpha} \,\psi(r)^{N-1} \, .
\end{equation}
Estimates \eqref{eq:stima-psi} and \eqref{est-psi-parts} plus the fact that $ \psi(r) $ is nondecreasing entail (for all $ r > r_0 $)
$$
\begin{aligned}
Q(r) & \le \left( \int_{0}^{r_0} \psi(s)^{N-1} \, ds +  \dfrac{r^{\alpha} \, \psi(r)^{N-1}}{c(N-1)} \right)^{\frac{1}{p}}\left(\int_r^\infty \frac{1}{\psi(s)^{N-1}} \, ds \right)^{\frac 1 2} \\
& = \left( \dfrac{1}{r^{\alpha} \, \psi(r)^{N-1}}\int_{0}^{r_0} \psi(s)^{N-1} \, ds +  \dfrac{1}{c(N-1)} \right)^{\frac{1}{p}}\left[ \left(r^{\alpha} \, \psi(r)^{N-1}\right)^{\frac{2}{p}}\int_r^\infty \frac{1}{\psi(s)^{N-1}} \, ds \right]^{\frac 1 2} \\
& \le \left( \frac{1}{r^{\alpha} \, \psi(r)^{N-1}} \int_{0}^{r_0} \psi(s)^{N-1} \, ds +  \dfrac{1}{c(N-1)}  \right)^{\frac{1}{p}}\left(\int_r^\infty s^{\frac{2\alpha}{p}} \psi(s)^{-\frac{(N-1)(p-2)}{p}}  \, ds \right)^{\frac 1 2} \\
 & \le \left( \frac{e^{-\frac{c(N-1)}{1-\alpha} r^{1-\alpha}}}{\kappa^{\frac{p-2}{2}} r^{\alpha}} \int_{0}^{r_0} \psi(s)^{N-1} \, ds +  \dfrac{\kappa^{-\frac{(N-1)(p-2)}{2}}  }{c(N-1)}  \right)^{\frac{1}{p}}\left( \int_r^\infty s^{\frac{2\alpha}{p}} e^{-\frac{c(N-1)(p-2)}{(1-\alpha)p}s^{1-\alpha}} ds \right)^{\frac 1 2} ,
\end{aligned}
$$ 
whence $ \lim_{r \to \infty} Q(r) = 0 $. As a consequence, because $ Q(r) $ is smooth and positive in $(0,\infty)$, it necessarily admits a maximum at some $ \overline{r}>0 $, which is a critical point. Since 
$$ 
Q^\prime(r)  = \frac{\psi(r)^{N-1}}{p}  \left(\int_0^r\psi(s)^{N-1} \, ds\right)^{\frac{1}{p}-1} \left(\int_r^\infty \frac{1}{\psi(s)^{N-1}} \, ds \right)^{\frac 1 2} \left( 1 -  \frac{p\int_0^{{r}}\psi(s)^{N-1} \, ds}{2\,\psi({r})^{2N-2}\int_r^\infty \frac{1}{\psi(s)^{N-1}} \, ds} \right) 
$$ 
at $ r=\overline{r} $ we find the identities
\begin{equation}\label{id-q-crit}
\int_{\overline{r}}^\infty \frac{1}{\psi(s)^{N-1}} \, ds  = \frac{p\int_0^{\overline{r}}\psi(s)^{N-1} \, ds}{2\,\psi(\overline{r})^{2N-2}} \qquad \Longrightarrow \qquad Q(\overline{r}) = \frac{\left( \frac{p}{2} \right)^{\frac 12} }{\psi(\overline{r})^{N-1}}  \left(\int_0^{\overline{r}}\psi(s)^{N-1} \, ds\right)^{\frac{p+2}{2p}} .
\end{equation} 
In particular, we can infer that
\begin{equation}\label{eq:sup-ineq}
\sup_{r \in (0,\infty)} Q(r) \le \left( \frac{p}{2} \right)^{\frac 12} \sup_{r \in (0,\infty)} \frac{1}{\psi(r)^{N-1}}  \left(\int_0^{r}\psi(s)^{N-1} \, ds\right)^{\frac{p+2}{2p}} .
\end{equation} 
Let us focus on the r.h.s.~of \eqref{eq:sup-ineq}. First of all note that the first two inequalities in \eqref{eq:curv}  yield 
\begin{equation}\label{eq:sup-2} 
\sup_{r \in (0,r_0)} \frac{1}{\psi(r)^{N-1}}  \left(\int_0^{r}\psi(s)^{N-1} \, ds\right)^{\frac{p+2}{2p}}\le r_0^{\frac{(N-2)(2^{\ast} - p)}{2p}} ,
\end{equation}
where in the case $ N=2 $ we mean $ (N-2)(2^\ast-p) = 4 $. On the other hand, by exploiting \eqref{eq:stima-psi}, \eqref{est-psi-parts}, \eqref{eq:sup-2} and the fact that $ \psi $ is nondecreasing, we obtain:
\begin{equation}\label{eq:sup-3} 
\begin{aligned}
 & \sup_{r \in (r_0,\infty)} \frac{1}{\psi(r)^{N-1}}  \left(\int_0^{r}\psi(s)^{N-1} \, ds \right)^{\frac{p+2}{2p}} \\
 \le & \sup_{r \in (r_0,\infty)} \left( \psi(r_0)^{-\frac{2p(N-1)}{p+2}}  \int_0^{r_0}\psi(s)^{N-1} \, ds +  \psi(r)^{-\frac{2p(N-1)}{p+2}}  \int_{r_0}^{r}\psi(s)^{N-1} \, ds \right)^{\frac{p+2}{2p}} \\
 \le & \sup_{r \in (r_0,\infty)} \left( r_0^{\frac{(N-2)(2^{\ast} - p)}{p+2}} + \frac{r^\alpha \, \psi(r)^{-\frac{(N-1)(p-2)}{p+2}}}{c(N-1)} \right)^{\frac{p+2}{2p}} \\
 \le & \sup_{r \in (r_0,\infty)} \left( r_0^{\frac{(N-2)(2^{\ast} - p)}{p+2}} + \frac{r^\alpha \, \kappa^{\frac{-(N-1)(p-2)}{p+2}} \, e^{ -\frac{c(N-1)(p-2)}{(p+2)(1-\alpha)} \, r^{1-\alpha}} }{ c(N-1)} \right)^{\frac{p+2}{2p}} \\
\le & \left( r_0^{\frac{(N-2)(2^{\ast} - p)}{p+2}} + \frac{[\alpha(p+2)]^{\frac{\alpha}{1-\alpha}}}{[c(N-1)]^{\frac{1}{1-\alpha}}\,(p-2)^{\frac{\alpha}{1-\alpha}}} \, r_0^{-\frac{(N-1)(p-2)}{p+2}} \, e^{\frac{c(N-1)(p-2)}{(p+2)(1-\alpha)}\,r_0^{1-\alpha}-\frac{\alpha}{1-\alpha}} \right)^{\frac{p+2}{2p}} ,
\end{aligned}
\end{equation}  
where we have computed explicitly the last supremum in \eqref{eq:sup-3} (over the whole $ \mathbb{R}^+ $) recalling the definition of $\kappa$. Hence, by combining \eqref{eq:sup-ineq}, \eqref{eq:sup-2} and \eqref{eq:sup-3}, estimate \eqref{cost} readily follows upon letting $ p \downarrow 2 $ and (in the case $ N=2 $) $ p \to \infty $. 
\end{proof}

We are now in position to prove Theorem \ref{teo}. 

\begin{proof}[Proof of Theorem \ref{teo}]
We consider the case $ p<2^\ast $ only, as it is already known that the Euclidean Sobolev inequality holds on any Cartan-Hadamard manifold. Let us first establish \eqref{tsfinale} and then show optimality according to \eqref{tsfinale-opt}. To our purposes, we introduce the function
\begin{equation}\label{psi-ast}
\psi_\star(r) := \left(\dfrac{\meas(S_{r})}{\omega_{N-1}}\right)^{\frac{1}{N-1}} \qquad \forall r \ge 0 \, ,
\end{equation} 
where $ \omega_{N-1} $ stands for the total surface measure of the unit sphere $ \mathbb{S}^{N-1} $. It is an elementary fact that $ \psi_\star \in \mathcal{A} $. Indeed, recalling \eqref{def-meas} and the regularity of $ A(r,\theta) $, we have $ \psi_\star \in C^\infty((0,\infty)) $. Furthermore, since $ \meas(S_{r}) / r^{N-1} \to \omega_{N-1} $ as $ r \to 0 $, we easily deduce that $ \psi_\star \in C^1([0,\infty)) $, $ \psi_\star(0)=0 $ and $ \psi_\star^\prime(0)=1 $. We aim at showing that $ \psi_\star $ fulfills the hypotheses \eqref{eq:curv} of Lemma \ref{lem-psi} for some positive $ r_0 \equiv r_0(\beta,C_0,R_0) $, $ c \equiv c(N,\beta,C_0,R_0) $ and $ \alpha=\beta/2 $. Thanks to \eqref{lap-m}, \eqref{lap-euc} and Lemma \ref{lem-GMV}, the following inequalities hold: 
\begin{equation}\label{stimasum-2}
\frac{\frac{\partial}{\partial r} A(r,\theta)}{A(r,\theta)} \ge c \, r^{-\frac{\beta}{2}} \qquad \forall (r,\theta) \in [r_0,\infty) \times \mathbb{S}^{N-1} \setminus \mathcal{P}  \, ,
\end{equation} 
\begin{equation}\label{stimasum-1}
\frac{\frac{\partial}{\partial r} A(r,\theta) }{A(r,\theta)} \geq \dfrac{N-1}{r} \qquad \forall (r,\theta) \in (0,\infty) \times \mathbb{S}^{N-1} \setminus \mathcal{P} \,,
\end{equation} 
for suitable constants $ c,r_0 >0 $ as above. By integrating \eqref{stimasum-1} and using again the asymptotic behavior of $ r \mapsto  \meas(S_{r}) $ as $ r \to 0 $, we find the well-known bound 
\begin{equation*}\label{eps-S}
\meas(S_{r})\geq \omega_{N-1} \, r^{N-1} \qquad \Longrightarrow \qquad \psi_\star(r) \ge r \qquad \forall r \ge 0 \, .
\end{equation*}
The fact that $\psi_\star^\prime(r) \ge 0 $ everywhere is a trivial consequence of \eqref{stimasum-1}, so we are left with establishing the rightmost inequality in \eqref{eq:curv}. To this aim, note that \eqref{stimasum-2} entails
\begin{equation}\label{psi-ast-e2}
\frac{d}{dr} \, \meas(S_r) =  \int_{\mathbb{S}^{N-1}} \frac{\partial}{\partial r} A(r,\theta) \, d\theta \ge c \, r^{-\frac{\beta}{2}} \int_{\mathbb{S}^{N-1}} A(r,\theta) \, d\theta = c \, r^{-\frac{\beta}{2}} \, \meas(S_r)  \qquad  \forall r \ge r_0 \, ,
\end{equation}
whence
\begin{equation}\label{psi-ast-e3}
\frac{\psi^\prime_\star(r)}{\psi_\star(r)} \ge \frac{c}{N-1} \, r^{-\frac{\beta}{2}} \qquad \forall r \ge r_0 \, ,
\end{equation}
namely the claimed inequality with $ \alpha = \beta/2 $, upon relabeling $c$. We have thus proved that the function $ \psi_\star $ defined in \eqref{psi-ast} satisfies all the assumptions of Lemma \ref{lem-psi}. As a result, we have 
\begin{equation}\label{cost-thm} 
\sup_{r\in(0,\infty)} \left(\int_0^r\psi_\star(s)^{N-1} \, ds\right)^{\frac{1}{p}} \left(\int_r^\infty \frac{1}{\psi_\star(s)^{N-1}} \, ds \right)^{\frac 1 2}\leq \frac{C \, p^{\frac{2+\beta}{2(2-\beta)}}}{(p-2)^{\frac{\beta}{2-\beta}}} 
\end{equation}
for a suitable $C\equiv C(N,\beta,C_0,R_0)>0 $. Thanks to \eqref{cost-thm}, we can apply Proposition \ref{thm:ko} with $ w(r) = \psi_\star(r)^{N-1} $, which ensures the validity of the Sobolev-type inequality
\begin{equation}\label{sob-ineq-1} 
\begin{gathered}
\left( \int_0^\infty \left|g(r)\right|^p \, \psi_\star(r)^{N-1} \, dr \right)^{\frac 1p} \le \left( 1+\tfrac p2 \right)^{\frac 1p} \left(  1+\tfrac 2p \right)^{\frac 12} \frac{C \, p^{\frac{2+\beta}{2(2-\beta)}}}{(p-2)^{\frac{\beta}{2-\beta}}} \left( \int_0^\infty \left| g^\prime(r) \right|^2 \, \psi_\star(r)^{N-1} \, dr \right)^{\frac 12} \\
\forall  g \in C^1_c([0,\infty)) \, .
\end{gathered}
\end{equation}
In order to pass from \eqref{sob-ineq-1} to \eqref{tsfinale}, it is enough to observe that $ f \in C^1_{c:\mathrm{rad}}(\mathbb{M}^N) $ implies $ f \equiv  f(r) \in C^1_c([0,\infty)) $ and the validity of the following identities:
\begin{equation*}\label{sob-ineq-2}
\begin{aligned}
\left\| f \right\|_{L^p\left(\mathbb{M}^N\right)} = & \left( \int_0^\infty \int_{\mathbb{S}^{N-1}} \left| f(r) \right|^p A(r,\theta)\, d\theta \, dr \right)^{\frac 1p} = \left( \omega_{N-1} \int_0^\infty \left| f(r) \right|^p \psi_\star(r)^{N-1} \, dr \right)^{\frac 1p}  , \\ 
 \left\| \nabla f \right\|_{L^2\left(\mathbb{M}^N\right)} = & \left(  \int_0^\infty \int_{\mathbb{S}^{N-1}} \left| f^\prime(r) \right|^2 A(r,\theta)\, d\theta \, dr \right)^{\frac 12} =  \left( \omega_{N-1} \int_0^\infty \left| f^\prime(r) \right|^2 \psi_\star(r)^{N-1} \, dr \right)^{\frac 12} , \\
\end{aligned} 
\end{equation*}
so that \eqref{tsfinale} is established upon relabeling $C$. 
Let us finally deal with optimality. To this end, take any function $ \psi \in \mathcal{A} $ such that 
\begin{equation}\label{c1-opt} 
\psi^{\prime\prime}(r) \ge 0 \quad \forall r > 0 \, ,   \qquad  \frac{\psi^{\prime\prime}(r)}{\psi(r)} = C_0 \, r^{-\beta}  \quad \forall r \ge R_0 \, ,
\end{equation}
which ensures that the associated model manifold $ \mathbb{M}^N_\psi $ is Cartan-Hadamard and complies with \eqref{sezionale}. Indeed, arguing as in \cite[Lemma 4.1]{GMV}, it is not difficult to prove that \eqref{c1-opt} implies
\begin{equation}\label{c1-opt-bis}
\frac{\psi^{\prime}(r)}{\psi(r)} \sim \sqrt{C_0} \, r^{-\frac{\beta}{2}}  \qquad \text{as } r \to \infty \, ,
\end{equation} 
where by $ a(r) \sim b(r) $ we mean that the ratio $ a(r)/b(r) $ tends to $1$. In particular, \eqref{c1-opt-bis} entails 
\begin{equation}\label{c1-opt-ter}
\frac{\psi^{\prime}(r)}{\psi(r)} \le 2 \, \sqrt{C_0} \, r^{-\frac{\beta}{2}}  \quad \forall r \ge r_0  \qquad \Longrightarrow \qquad \psi(r) \le c_1 \, e^{c_2 \, r^{\frac{2-\beta}{2}} } \quad \forall r \ge r_0 \, ,
\end{equation} 
$ r_0 \ge 1 $ and $c_1,c_2 > 0  $ being suitable constants that depend on $ \psi $, whose exact values are not relevant to our purposes. Hence, \eqref{c1-opt-ter} plus a simple integration by parts yield
\begin{equation}\label{c2-opt}
\int_r^\infty \frac{1}{\psi(s)^{N-1}} \, ds \ge \frac{r^{\frac{\beta}{2}}}{2\sqrt{C_0}(N-1) \, \psi(r)^{N-1}} \qquad \forall r \ge r_0 \, .
\end{equation} 
Similarly we obtain
\begin{equation*}\label{c2-opt-bis}
\int_{r_0}^r \psi(s)^{N-1} \, ds \ge \frac{1}{2\sqrt{C_0}(N-1)} \left(  r^{\frac{\beta}{2}} \, \psi(r)^{N-1} - r_0^{\frac{\beta}{2}} \, \psi(r_0)^{N-1} - \frac{\beta}{2} \, \int_{r_0}^r \frac{\psi(s)^{N-1}}{s^{\frac{2-\beta}{2}}} \, ds \right) \quad \forall r \ge r_0 \, ,
\end{equation*} 
which entails (recall that $ \beta \le 2 $)
\begin{equation}\label{c2-opt-quater}
\int_{r_0}^r \psi(s)^{N-1} \, ds \ge \frac{r^{\frac{\beta}{2}} \, \psi(r)^{N-1}}{4\sqrt{C_0}(N-1)+2} \qquad \forall r \ge \hat{r}_0 \, ,
\end{equation}  
provided $ \hat{r}_0 $ is selected, for instance, so as to satisfy $ \psi(\hat{r}_0)^{N-1} =2 \, \psi(r_0)^{N-1} $. By combining \eqref{c1-opt-ter}, \eqref{c2-opt} and \eqref{c2-opt-quater}, we deduce that
\begin{equation}\label{opt-11}
\left(\int_0^r \psi(s)^{N-1} \, ds \right)^{\frac{1}{p}} \left(\int_r^\infty \frac{1}{\psi(s)^{N-1}} \, ds \right)^{\frac 1 2} \ge C \, \frac{r^{\frac{\beta(p+2)}{4p}}}{\psi(r)^{\frac{(N-1)(p-2)}{2p}}}  \ge C \, \frac{r^{\frac{\beta(p+2)}{4p}}}{e^{c_2 \, \frac{(N-1)(p-2)}{2p} \, r^{\frac{2-\beta}{2}}}} \quad \forall r \ge \hat{r}_0 \, ,
\end{equation}
where from here on $ C $ denotes a general positive constant that can be taken independent of $ p \in (2,2^\ast] \cap (2,\infty) $, which will not be relabeled. A straightforward calculation shows that the maximum over $ r \in (0,\infty) $ of the rightmost term in \eqref{opt-11} is attained at 
\begin{equation*}\label{opt-12}
\overline{r} = \left[ \frac{ \beta (p+2)}{c_2(N-1)(2-\beta)(p-2)} \right]^{\frac{2}{2-\beta}} ,
\end{equation*} 
which ensures that
\begin{equation}\label{opt-13}
\sup_{r \in (0,\infty)} \left(\int_0^r \psi(s)^{N-1} \, ds \right)^{\frac{1}{p}} \left(\int_r^\infty \frac{1}{\psi(s)^{N-1}} \, ds \right)^{\frac 1 2} \ge C \left( \frac{p+2}{p-2} \right)^{\frac{\beta}{2-\beta}}
\end{equation} 
under the constraint
\begin{equation}\label{opt-14}
\left[ \frac{ \beta (p+2)}{c_2(N-1)(2-\beta)(p-2)} \right]^{\frac{2}{2-\beta}}  \ge \hat{r}_0 \, .
\end{equation}
It is apparent that \eqref{opt-13}--\eqref{opt-14}, together with \eqref{sob-one-2}, yield \eqref{tsfinale-opt} at least for $ N \ge 3 $, where $ 2^\ast < \infty $. We are left with the establishing the correct behavior as $ p \to \infty $ in the case $ N=2 $. To this aim note that, as a simple consequence of the fact that $ \psi(r) \sim r $ as $ r \to 0 $, we have 
\begin{equation*}\label{opt-15}
\sup_{r \in (0,\infty)} \left(\int_0^r \psi(s) \, ds \right)^{\frac{1}{p}} \left(\int_r^\infty \frac{1}{\psi(s)} \, ds \right)^{\frac 1 2} \ge C \sup_{r \in \left(0, \, e^{-2}\right)} \, r^{\frac 2 p} \left(-\log r \right)^{\frac 12}  \ge C \, p^{\frac{1}{2}} \, ,
\end{equation*} 
which, upon exploiting again \eqref{sob-one-2}, also ensures the validity of \eqref{tsfinale-opt} as $ p \to \infty $. 
\end{proof} 

\subsection{The quasi-Euclidean case}\label{sec:qe} 

Similarly to Subsection \ref{sec:sub-hyp} we first show that, for appropriate weights $ w(r) = \psi(r)^{N-1} $ associated with ``quasi-Euclidean'' model manifolds, one can bound in a quantitative way the supremum appearing in \eqref{sob-one-1}.

\begin{lemma}\label{lem-psi2}
Let $ \psi \in \mathcal{A} $ satisfy the following assumptions: 
\begin{equation}\label{eq:curv2}
\psi(r) \ge r \quad \forall r \ge 0 \, , \quad \psi^\prime(r) \ge 0 \quad \forall r \ge 0  \, , \quad \dfrac{\psi'(r)}{\psi(r)} \geq \dfrac{c}{r}-\dfrac{c'}{r^q} \quad \forall r \ge r_0 \, ,
\end{equation} 
for some $ c>1 $, $c'>0$, $q>1$ and $ r_0>0 $. Let $ N \in \mathbb{N} $ with $ N \ge 2 $. Then 
\begin{equation}\label{cost2}
\sup_{r\in(0,\infty)} \left(\int_0^r\psi(s)^{N-1} \, ds\right)^{\frac{1}{p}}\left(\int_r^\infty \frac{1}{\psi(s)^{N-1}} \, ds \right)^{\frac 1 2} \le C \, \sqrt{p} \qquad \forall p\in\left[ \tfrac{2\tilde{N}}{\tilde{N}^{\phantom{a}}\!\!\!-2} , 2^\ast \right) ,
\end{equation}
where $\tilde{N} \equiv \tilde{N}(N,c):=c(N-1)+1$ and $C$ is a positive constant depending only on $ N, c, c^\prime, q, r_0 $. 
\end{lemma}
\begin{proof}
We will not provide full details (for which we refer to \cite[Lemma 4.3]{R}), since the strategy follows closely the lines of proof of Lemma \ref{lem-psi}. Let $ Q(r) $ be defined by \eqref{eq:def-Q}.
By integrating the last inequality in \eqref{eq:curv2} (and taking advantage of the first one as well), we obtain:
\begin{equation}\label{eq:stima-psi2}
\psi(r)\geq \kappa \, r^c \quad \forall r \ge r_0 \, , \qquad \text{where } \, \kappa \equiv  \kappa(c,c',q,r_0) > 0 \, ,
\end{equation}
which ensures that $ Q(r) $ is a smooth function of $r>0$, given the finiteness of the integrals involved. Moreover, because $ \psi(r) \sim r $ as $ r \to 0 $ and $ p < 2^\ast $, it is immediate to check that $ \lim_{r \to 0} Q(r) = 0 $. In order to deal with the limit at infinity, we need again some integral bounds. Still the rightmost inequality in \eqref{eq:curv2} yields 
\begin{equation}\label{ee5}
\psi(r)^{N-1}\leq\dfrac{2r}{(c+1)(N-1)}\dfrac{d}{dr}\left( \psi^{N-1} \right)\!(r) \quad \forall r \ge \hat{r}_0 \, , \qquad \text{for some } \, \hat{r}_0 \equiv \hat{r}_0(c,c^\prime,q,r_0) > r_0  \, .
\end{equation}
Upon integrating by parts \eqref{ee5} between $ \hat{r}_0 $ and $ r>\hat{r}_0 $, we deduce that
\begin{equation}\label{est-psi-parts2}
\frac{(c+1)(N-1)}{2}\int_{\hat{r}_0}^r \psi(s)^{N-1} \, ds  \le  r \, \psi(r)^{N-1} - \hat{r}_{0} \, \psi(\hat{r}_{0})^{N-1} - \int_{\hat{r}_0}^r \, \psi(s)^{N-1} \, ds  \leq  r \, \psi(r)^{N-1}  \, .
\end{equation}
In a similar way, we can infer that 
\begin{equation*} 
\frac{(c+1)(N-1)}{2} \int_r^\infty\dfrac{1}{\psi(s)^{N-1}} \, ds \leq \dfrac{r}{\psi(r)^{N-1}}+\int_r^\infty\dfrac{1}{\psi(s)^{N-1}} \, ds \qquad \forall r \ge \hat{r}_0  \, ,
\end{equation*}
whence 
\begin{equation}\label{est-psi-part3}
\int_r^\infty\dfrac{1}{\psi(s)^{N-1}}\, ds\leq \dfrac{2}{(c-1)(N-1)+2(N-2)} \, \dfrac{r}{\psi(r)^{N-1}} \qquad \forall r \ge \hat{r}_0 \, .
\end{equation} 
By plugging \eqref{est-psi-parts2} in \eqref{eq:def-Q}, exploiting \eqref{est-psi-part3}, the fact that $ \psi'(r)\ge 0 $ and \eqref{eq:stima-psi2}, we obtain: 
\begin{equation*}\label{stimaQlemma4}
\begin{aligned}
Q(r) & \le \left( \frac{1}{r \, \psi(r)^{N-1}} \int_{0}^{\hat{r}_0} \psi(s)^{N-1} \, ds +  \dfrac{2}{(c+1)(N-1)}  \right)^{\frac{1}{p}}\left[ \left(r \, \psi(r)^{N-1} \right)^{\frac{2}{p}} \int_r^\infty \frac{1}{\psi(s)^{N-1}} \, ds \right]^{\frac 1 2} \\
& \le \left( \frac{1}{r \, \psi(r)^{N-1}} \int_{0}^{\hat{r}_0} \psi(s)^{N-1} \, ds +  \dfrac{2}{(c+1)(N-1)}  \right)^{\frac{1}{p}}\left[ \frac{2 \, r^{\frac{p+2}{p}} \psi(r)^{-\frac{(N-1)(p-2)}{p}} }{(c-1)(N-1)+2(N-2)} \right]^{\frac 1 2} \\
& \le C \, r^{\frac{p+2}{2p}-\frac{c(N-1)(p-2)}{2p}} 
\end{aligned}
\end{equation*}
for all $ r > \hat{r}_0 $, where from here on $ C $ stands for a general positive constant depending only on $ N, c, c^\prime, q, r_0 $ (that we will not relabel). In particular, we have that
\begin{equation*}\label{lss}
 p \geq \tfrac{ 2\tilde{N} }  {  \tilde{N}^{ \phantom{a}^{  \phantom{a}^{ \phantom{a} } }} \!\!\!\!\!\!\!\!\!  -2  }  \qquad \Longleftrightarrow \qquad {p+2}-{c(N-1)(p-2)} \leq 0  \qquad \Longrightarrow \qquad \limsup_{r \to \infty} Q(r) \le C \, .
\end{equation*}
There are two possibilities: either $Q(r)$ does not admit a maximum, in which case we can claim that $ Q(r) < C $ for all $ r>0 $ and hence \eqref{cost2} is established, or it admits a maximum at some $ \overline{r}>0  $. In the latter case, by carrying out the same computations as in the proof of Lemma \ref{lem-psi}, formulas \eqref{id-q-crit}, \eqref{eq:sup-ineq} and \eqref{eq:sup-2} (with $ r_0 $ replaced by $ \hat{r}_0 $) are still true. As a result, upon exploiting also \eqref{eq:stima-psi2}, \eqref{est-psi-parts2} and the fact that $ \psi $ is nondecreasing, we deduce the following: 
\begin{equation}\label{eq:sup-3bis}
\begin{aligned}
 & \sup_{r \in (\hat{r}_0,\infty)} \frac{1}{\psi(r)^{N-1}}  \left(\int_0^{r}\psi(s)^{N-1} \, ds\right)^{\frac{p+2}{2p}} \\
 \le & \sup_{r \in (\hat{r}_0,\infty)} \left( \psi(\hat{r}_0)^{-\frac{2p(N-1)}{p+2}} \int_0^{\hat{r}_0}\psi(s)^{N-1} \, ds +  \psi(r)^{-\frac{2p(N-1)}{p+2}}  \int_{\hat{r}_0}^{r}\psi(s)^{N-1} \, ds \right)^{\frac{p+2}{2p}} \\
 \le & \sup_{r \in (\hat{r}_0,\infty)} \left( \hat{r}_0^{\frac{(N-2)(2^{\ast} - p)}{p+2}} + \frac{2r\,\psi(r)^{-\frac{(N-1)(p-2)}{p+2}}}{(c+1)(N-1)}   \right)^{\frac{p+2}{2p}} \\
\le & \left( \hat{r}_0^{\frac{(N-2)(2^{\ast} - p)}{p+2}} + C \, \hat{r}_0^{1-\frac{c(N-1)(p-2)}{p+2}}\right)^{\frac{p+2}{2p}} .
\end{aligned}
\end{equation}  
Note that for $ N=2 $ we still mean $ (N-2)(2^\ast-p) = 4 $. The validity of  \eqref{cost2} is then a consequence of \eqref{eq:sup-ineq}, \eqref{eq:sup-2} (with $ r_0 $ replaced by $ \hat{r}_0 $) and \eqref{eq:sup-3bis}, up to relabeling $C$. 
\end{proof}

We can finally prove Theorem \ref{teo2}. 

\begin{proof}[Proof of Theorem \ref{teo2}]
We aim at showing that the function $ \psi_\star $ defined by \eqref{psi-ast}
satisfies the hypotheses of Lemma \ref{lem-psi2}. The first two inequalities of \eqref{eq:curv2}, along with the fact that $ \psi_\star \in \mathcal{A} $, follow by reasoning as in the proof of Theorem \ref{teo}. As concerns the last one, some adaptations have to be performed: we mainly refer to \cite[Subsection 8.1]{GMV}. First of all, we observe that the general solution of the differential equation 
\begin{equation*}\label{ee54}
\phi''(r) = C_1 \, r^{-2} \, \phi(r)  \qquad \forall  r \in \mathbb{R}^+ 
\end{equation*}
is explicit, namely
\begin{equation}\label{ee55}
\phi(r)=a_1 \, r^{q_1} + a_2 \, r^{q_2} \qquad \forall r \in \mathbb{R}^+
\end{equation}
for arbitrary real constants $a_1$ and $a_2$, where $q_{1,2}=(1\pm\sqrt{1+4C_1})/2$. It is not difficult to show (by arguing as in \cite[Subsection 8.1]{GMV}) that one can construct a function $ \psi \in \mathcal{A} $ such that 
\begin{equation*}\label{ee86}
\begin{cases}
\psi''(r) = 0 & \forall r \in [0,R_0] \, , \\
\psi''(r) \le C_1 \, r^{-2} \, \psi(r)  &  \forall r \in (R_0,2R_0) \, , \\ 
\psi''(r) = C_1 \, r^{-2} \, \psi(r) &  \forall r \ge 2 R_0  \, , \\
\end{cases}
\end{equation*}
which therefore complies with \eqref{ee55} for every $ r \ge 2R_0 =: r_0 $ and constants $ a_1 > 0 $, $ a_2 \in \mathbb{R} $ depending only on $C_1,R_0$. In view of \eqref{sezionale2}, we are in position to apply the Laplacian-comparison results of Subsection \ref{lc} (specifically \eqref{comp-sect-2}), guaranteeing that 
\begin{equation}\label{stimasum}
\mathsf{m}(r,\theta) \geq (N-1) \, \dfrac{\psi'(r)}{\psi(r)} \geq (N-1) \left(\dfrac{q_1}{r}-\dfrac{h}{r^{1+\sqrt{1+4C_1}}}\right) \qquad \forall (r,\theta) \in [r_0,\infty) \times \mathbb{S}^{N-1}  \, ,
\end{equation}
where $h$ is a suitable positive constant depending on $ a_1,a_2,q_1,q_2,r_0 $. Thanks to \eqref{stimasum}, by proceeding as in \eqref{psi-ast-e2}--\eqref{psi-ast-e3} we end up with
$$
\frac{\psi_\star'(r)}{\psi_\star(r)} \ge \dfrac{q_1}{r}-\dfrac{h}{r^{1+\sqrt{1+4C_1}}} \qquad \forall r\geq r_{0} \, ,
$$
so that the function $ \psi_\star $ satisfies the hypotheses of Lemma \ref{lem-psi2} with $c=q_1$, $c^\prime=h$ and $ q = 1+\sqrt{1+4C_1} $. As a consequence, we deduce that
\begin{equation}\label{sup esplicito 2}
\sup_{r\in(0,\infty)} \left(\int_0^r \psi_\star(s)^{N-1} \, ds\right)^{\frac{1}{p}}\left(\int_r^\infty \dfrac{1}{\psi_\star(s)^{N-1}} \, ds \right)^{\frac 1 2} \le {C} \, \sqrt{p} \qquad \forall p \in \left[\tilde{2}, 2^\ast\right) ,
\end{equation}
where $ \tilde{2} $ is defined in \eqref{tsfinaleX} and $C$ is a positive constant as in the statement. In view of \eqref{sup esplicito 2} the thesis follows as in the proof of Theorem \ref{teo} upon applying Proposition \ref{thm:ko}. We finally establish optimality. To this aim, it is enough to consider any function $\psi\in\mathcal{A}$ such that 
\begin{equation*}\label{opt-77}
\psi''(r) \ge 0 \quad \forall r>0 \qquad \text{and} \qquad \psi(r) = r^{\frac{\tilde{N}-1}{N-1}} \quad \text{for large } r \, ,
\end{equation*} 
where $ \tilde{N} $ is related to $ C_1 $ through \eqref{tsfinaleX}. This makes sure that the associated model manifold $\mathbb{M}^N_\psi$ is Cartan-Hadamard and complies with \eqref{sezionale2}. Recalling that $\tilde{N}-2>0$, we therefore obtain
\begin{equation*}\label{ee101}
\left(\int_r^\infty \dfrac{1}{\psi(s)^{N-1}}\, ds\right)^\frac{1}{2} = \frac{r^{-\frac{\tilde{N}-2}{2}}}{\sqrt{\tilde{N}-2}} \qquad \text{and} \qquad \left(\int_{0}^r \psi(s)^{N-1} \, ds\right)^\frac{1}{p} \ge \frac{r^{\frac{\tilde{N}}{p}}}{\big(2\tilde{N}\big)^{\frac1p}}    \qquad \text{for large } r \, ,
\end{equation*}
whence
\begin{equation}\label{e103}
\left(\int_{0}^r \psi(s)^{N-1} \, ds\right)^\frac{1}{p} \left(\int_r^\infty \dfrac{1}{\psi(s)^{N-1}} \, ds\right)^\frac{1}{2} \ge \frac{r^\frac{2 \tilde N - p \tilde N + 2p}{2p}}{\sqrt{\tilde{N}-2}\,\big(2\tilde{N}\big)^{\frac1p}}   \qquad \text{for large } r \, .
\end{equation}
Clearly the r.h.s.~of \eqref{e103} stays bounded as $ r \to \infty $ if and only if $ 2 \tilde N - p \tilde N + 2p \le 0 $, namely $ p \ge \tilde{2} $. Hence, thanks to Proposition \ref{thm:ko}, we can assert that in this case \eqref{tsfinale2} fails for all $ p \in [2,\tilde{2}) $. As concerns the behavior of the optimal constant as $ p \to \infty $ (for $N=2$), one reasons exactly as in the end of the proof of Theorem \ref{teo}.
\end{proof}

\section{Failure of the inequalities for nonradial functions}\label{sect:nonrad} 

In this section we prove Theorem \ref{Controex teo} by constructing an explicit sequence of nonradial functions that make the Rayleigh quotient associated with inequality \eqref{sob-one-nnrad proof} blow up.

\begin{proof}[Proof of Theorem \ref{Controex teo}] 
Since the Ricci curvature of $ \mathbb{M}^N $ is everywhere nonpositive, in view of \eqref{ricci-chch tesi} there exists a nondecreasing and positive function function $ G:\mathbb{R}^+ \to \mathbb{R}^+ $, with $  \lim_{R \to \infty} G(R) = \infty  $, such that for every $ R>0 $ we have 
\begin{equation}\label{bound-uniform-ricci}
\mathrm{Ric}(x) \ge - \frac{N-1}{G(R)^2} \qquad \forall x \in \mathbb{M}^N \setminus  B_R \, .
\end{equation}
Let us set
\begin{equation}\label{fR-dist proof} 
f_R(x) := \left( 1-\frac{\operatorname{d}(x,o_R)}{{G(R)}} \right)_+ \qquad \forall x \in \mathbb{M}^N \, ,
\end{equation}
$ o_R \in \mathbb{M}^N $ being any point that satisfies $ \operatorname{d}(o_R,o) = R + {G(R)}  $. Thanks to \eqref{bound-uniform-ricci}, we deduce that  
\begin{equation}\label{bound-uniform-R}
\mathrm{Ric}(x) \ge - \frac{N-1}{G(R)^2} \qquad \forall x \in B_{{G(R)}}(o_R) \, ,
\end{equation}
where $ B_r(o_R) $ stands for the geodesic ball of radius $r>0$ centered at $o_R$. Note that $ f_R $ is in fact only Lipschitz regular, but this is not an issue (recall the discussion in Subsection \ref{ext}). In view of \eqref{fR-dist proof}, it follows that
\begin{equation}\label{fR-dist-1 proof} 
\left| \nabla f_R(x) \right| = \frac{1}{{G(R)}} \, \chi_{B_{{G(R)}}(o_R)}(x) \qquad \text{and}  \qquad \left| f_R(x) \right| \ge \frac{1}{2} \, \chi_{B_{\frac{G(R)}{2}}(o_R)}(x) \qquad \forall x \in \mathbb{M}^N \, .
\end{equation} 
In particular, the $ L^p $ norm of $ f_R $ is readily estimated from below: 
\begin{equation}\label{Lp-below proof}
\left\| f_R \right\|_{L^p\left( \mathbb{M}^N \right)}^p = \int_{\mathbb{M}^N} \left| f_R \right|^p d\nu \ge \frac{1}{2^p} \, \nu\!\left( B_{\frac{G(R)}{2}}(o_R) \right) \ge \frac{\omega_{N-1}}{2^{p+N} \, N} \, G(R)^{N} \, ,
\end{equation}
where in the last inequality we have used the simple fact that, because $ \mathbb{M}^N $ is Cartan-Hadamard, the volume growth of balls (w.r.t.~any pole) is at least Euclidean. This can be seen, for instance, as a direct consequence of Laplacian comparison, due to \eqref{meas} and \eqref{lap-euc}. As concerns the $ L^2 $ norm of the gradient, from \eqref{fR-dist-1 proof} we have
\begin{equation}\label{L2-below proof}
\left\| \nabla f_R \right\|_{L^2\left( \mathbb{M}^N \right)}^2 = \int_{\mathbb{M}^N} \left| \nabla f_R \right|^2 d\nu = \frac{1}{G(R)^2}  \, \nu\!\left( B_{G(R)}(o_R) \right) .
\end{equation}
By virtue of \eqref{bound-uniform-R} and again Laplacian comparison (or more in general Bishop-Gromov, see \cite[Theorem 1.1]{Hebey}), we can assert that the volume of $ B_{G(R)}(o_R) $ is bounded from above by the volume of the geodesic ball of the same radius in the hyperbolic space of constant Ricci curvature $ -(N-1)/G(R)^2  $, which corresponds to the model function
$$
\psi(r)= G(R) \, \sinh\!\left(\frac{{r}}{G(R)}  \right) .
$$
As a result, 
\begin{equation}\label{ricci-BV proof} 
\nu\!\left( B_{G(R)}(o_R) \right)\le \omega_{N-1} \, G(R)^{N-1} \int_0^{G(R)} \sinh \! \left( \tfrac{r}{G(R)} \right)^{N-1} dr = \omega_{N-1} \, G(R)^N \int_0^{1} \sinh(s)^{N-1} \, ds \, .
\end{equation} 
By combining \eqref{L2-below proof} and \eqref{ricci-BV proof}, we obtain: 
\begin{equation}\label{L2-below-bis proof}
\left\| \nabla f_R \right\|_{L^2\left( \mathbb{M}^N \right)}^2 \le  \omega_{N-1} \int_0^{1} \sinh(s)^{N-1} \, ds \ G(R)^{N-2} =: \omega_{N-1} \, a_N \, G(R)^{N-2}   \, .
\end{equation} 
If \eqref{sob-one-nnrad proof} was true, thanks to \eqref{Lp-below proof} and \eqref{L2-below-bis proof} we would end up with 
$$
\left( \frac{\omega_{N-1}}{2^{p+N} \, N} \, G(R)^{N} \right)^{\frac 1p} \le C \left(  \omega_{N-1} \, a_N \, G(R)^{N-2}  \right)^{\frac{1}{2}} ,
$$
namely
$$
G(R)^{2 N \left(  \frac1p-\frac1{2^\ast} \right) } \le 4^{\frac{p+N}{p}} N^{\frac 2 p} \, \omega_{N-1}^{\frac{p-2}{p}} \, a_N \, , 
$$
and the contradiction is achieved upon letting $ R \to \infty $, since $ p < 2^\ast $ by assumption.
\end{proof}

\section{The porous medium equation on Cartan-Hadamard model manifolds}\label{PME-application} 

Theorem \ref{teo} has some interesting consequences concerning \emph{smoothing effects} for problem \eqref{PME}, at least when the initial datum belongs to $ L^1(\mathbb{M}^N) $, is radially symmetric w.r.t.~the pole $ o $ and the manifold at hand is a {model} (so that radiality is preserved along the evolution). Below we let $ L^1_{\mathrm{rad}}\big(\mathbb{M}^N_\psi\big) $ denote the space formed by all such functions. 

\begin{theorem}\label{thm: sobo-deg}
Let $ \mathbb{M}^N_\psi $ be a Cartan-Hadamard model manifold satisfying
\begin{equation}\label{sezionale-smooth} 
\mathrm{Sect}_\omega(x) = \frac{1}{N-1} \, \mathrm{Ric}_o(x) = - \frac{\psi^{\prime\prime}(r)}{\psi(r)} \leq - C_0 \, r^{-\beta} \qquad \forall x \in \mathbb{M}^N_\psi \setminus B_{R_0} \, ,
\end{equation}
for some $ \beta \in (0,2) $ and $C_0,R_0>0$. Then there exists $ K>0 $, depending only on $ m, N , \beta , C_0 , R_0$, such that for all $ u_0 \in L^1_{\mathrm{rad}}\big(\mathbb{M}^N_\psi\big) $ the solution $u$ of \eqref{PME} fulfills the smoothing estimate
\begin{equation}\label{eq: log-alfa}
\left\| u(\cdot,t) \right\|_{L^\infty(\mathbb{M}^N_\psi)} \le K \left[ \log\!\left(t \left\| u_0 \right\|_{L^1(\mathbb{M}^N_\psi)}^{m-1} + e \right) \right]^{\frac{2+\beta}{(m-1)(2-\beta)}}  t^{-\frac{1}{m-1}}  \qquad \forall t>0 \, .
\end{equation}
The result is optimal w.r.t.~long-time behavior, in the sense that if \eqref{sezionale-smooth} holds with reverse inequality 
then there exist initial data $ u_0 \in L^1_{\mathrm{rad}}\big(\mathbb{M}^N_\psi\big) $ for which \eqref{eq: log-alfa} holds with reverse inequality for large $t$. 
\end{theorem}

The strategy of proof follows the lines of \cite[Theorem 3.1]{GM16} and \cite[Theorem 3.2]{GMV}. For the reader's convenience here we write down a concise argument, with the purpose to allow one to realize how the Sobolev-type inequalities \eqref{tsfinale} come into play. 

\begin{proof}
Let $ q >0  $ and $ \sigma = p/2 \in (1 , 2^\ast/2 ]$, both being for the moment free (but fixed) parameters. We can suppose with no loss of generality that $ u_0 \in L^1\big(\mathbb{M}^N_\psi\big) \cap L^\infty\big(\mathbb{M}^N_\psi\big) $. In order to make rigorous the computations we will carry out, some approximation procedures are necessary, which however will be skipped since they are out of the scope of this section: we refer e.g.~to \cite{GM13,GMP13,GM16} for more details. To improve readability, throughout $ \| \cdot \|_q $ will stand for $ L^q\big(\mathbb{M}^N_\psi\big) $ norms. First of all, we multiply the differential equation in \eqref{PME} by $ u(\cdot,t)^q $, integrate by parts and use \eqref{tsfinale} (note that if $ u_0 $ is radial  so is $ u(\cdot,t) $). This yields
\begin{equation}\label{eq:log-gamma-1}
\frac{d}{d t} \|u(\cdot,t)\|_{q+1}^{q+1} = - \frac{4q(q+1)m}{(m+q)^2} \left\| \nabla\!\left(u^{\frac{q+m}{2}}\right)\!(\cdot,t) \right\|_2^2 \le -\frac{4q(q+1)m}{(m+q)^2\,C_\sigma^2} \, \|u(\cdot,t)\|_{\sigma(q+m)}^{q+m} \, ,
\end{equation}
where we set
\begin{equation}\label{eq:log-gamma-xx}
C_\sigma = C_{p/2} := \frac{C \, p^{\frac{2+\beta}{2(2-\beta)}}}{(p-2)^{\frac{\beta}{2-\beta}}} \, ,
\end{equation}
$C$ being the same constant as in \eqref{tsfinale}. Upon taking advantage of standard interpolation and the well-known fact that the $ L^1\big(\mathbb{M}^N_\psi\big) $ norm does not increase along the evolution, we infer that
\begin{equation}\label{eq:log-gamma-2}
\| u(\cdot,t) \|_{q+1} \le \| u(\cdot,t) \|_{\sigma(q+m)}^{\frac{\sigma(q+m)q}{[\sigma(q+m)-1](q+1)}} \, \| u_0 \|_{1}^{\frac{\sigma(q+m)-(q+1)}{[\sigma(q+m)-1](q+1)}} \qquad \forall t>0 \, .
\end{equation}
For simplicity, let us assume $ \| u_0 \|_1=1 $ (the general case can be handled by a routine time-scaling argument). Hence, as a consequence of \eqref{eq:log-gamma-1} and \eqref{eq:log-gamma-2}, it follows that
\begin{equation}\label{eq:log-gamma-3}
\frac{d}{dt} \|u(\cdot,t)\|_{q+1}^{q+1} \le -\frac{4q(q+1)m}{(m+q)^2\,C_\sigma^2} \, \|u(	\cdot,t)\|_{q+1}^{(q+1)\frac{\sigma(q+m)-1}{\sigma q}} \, .
\end{equation}
The integration of \eqref{eq:log-gamma-3} entails 
$$ 
y(t)^{\frac{\sigma m -1}{\sigma q}} \le \left[ {y(0)^{-\frac{\sigma m -1}{\sigma q}}} + \tfrac{4m(q+1)(\sigma m -1)}{\sigma (q+m)^2 \, C_\sigma^2 }  \, t \right]^{-1}    \quad \forall t>0 \, , \qquad   y(t):=\| u(\cdot,t) \|_{q+1}^{q+1} \, , 
$$
whence 
\begin{equation}\label{eq:log-gamma-4}
\| u(\cdot,t) \|_{q+1} \le \left[ \frac{\sigma (q+m)^2 \, C_\sigma^2 }{4m(q+1)(\sigma m -1)} \right]^{\frac{\sigma q}{(q+1)(\sigma m -1)}} t^{-\frac{\sigma q}{(q+1)(\sigma m -1)}} \qquad \forall t>0 \, .
\end{equation}
By previous results (see e.g.~\cite[Corollary 5.6]{GM13} or \cite[Theorem 1.5]{BG05}), the validity of \eqref{tsfinale} for a fixed $ p/2 = \sigma_0 \in  (1 , 2^\ast/2 ) $ implies the smoothing estimate 
\begin{equation}\label{eq:log-gamma-5}
\left\| u(\cdot,t) \right\|_{\infty} \leq  K \, t^{-\frac{\sigma_0}{(\sigma_0-1)(q+1)+\sigma_0(m-1)}} \left\|  u_0 \right\|_{q+1}^{\frac{(\sigma_0-1)(q+1)}{(\sigma_0-1)(q+1)+\sigma_0(m-1)}} \qquad \forall t>0 \, ,
\end{equation}
where from here on we let $ K $ denote a general positive constant that depends on $ m, N , \beta , C_0 , R_0$ (which will not be relabeled). Therefore, the combination of \eqref{eq:log-gamma-4} (evaluated at time $t/2$) and \eqref{eq:log-gamma-5} with time origin shifted from $ 0 $ to $t/2$ (semigroup property) yields
\begin{equation}\label{eq:log-gamma-6}
\left\| u(\cdot,t) \right\|_{\infty} \leq  K \left[ \frac{\sigma (q+m)^2 \, C_\sigma^2 }{4m(q+1)(\sigma m -1)} \right]^{\frac{\sigma q ({\sigma_0}-1)}{ (\sigma m -1)[({\sigma_0}-1)(q+1)+{\sigma_0}(m-1) ] }}  \, t^{-\frac{{\sigma_0}(\sigma m -1)+\sigma q ({\sigma_0}-1)}{ (\sigma m -1)[({\sigma_0}-1)(q+1)+{\sigma_0}(m-1) ] }} 
\end{equation}
for all $ t>0 $. Because $ q>0 $ is a free parameter, we can let $ q =\log(t+e) $ in \eqref{eq:log-gamma-6} so as to obtain
\begin{equation}\label{eq:log-gamma-7}
\begin{aligned}
\left\| u(\cdot,t) \right\|_{\infty} \leq & \, K \left\{ \frac{\sigma  [\log(t+e)+m ]^2}{4m[\log(t+e)+1](\sigma m -1)} \right\}^{-\frac{\sigma({\sigma_0}m-1)}{(\sigma m-1)\{ ({\sigma_0}-1)[\log(t+e)+1] + {\sigma_0}(m-1) \}}} \\
& \times \, t^{\frac{{\sigma_0}-\sigma}{(\sigma m-1)\{({\sigma_0}-1)[\log(t+e)+1]+{\sigma_0}(m-1)\}}} \left\{ \frac{\sigma [1+m/\log(t+e)]^2}{4m[1+1/\log(t+e)](\sigma m -1)} \right\}^{\frac{\sigma}{\sigma m-1}}
& \\
& \times \, C_\sigma^{-\frac{2\sigma({\sigma_0}m-1)}{(\sigma m-1)\{ ({\sigma_0}-1)[\log(t+e)+1] + {\sigma_0}(m-1) \}}} \left[ \log(t+e) \, C_\sigma^2 \right]^{\frac{\sigma}{\sigma m-1}} t^{-\frac{\sigma}{\sigma m-1}} \qquad \forall t>0 \, .
\end{aligned}
\end{equation}
If $ \sigma \in (1,\sigma_0) $, it is apparent that the first two factors in the r.h.s.~of \eqref{eq:log-gamma-7} can be bounded from above by another general positive constant $ K $, whence
\begin{equation*}\label{eq:log-gamma-8}
\left\| u(\cdot,t) \right\|_{\infty} \leq  K \, C_\sigma^{-\frac{2\sigma({\sigma_0}m-1)}{(\sigma m-1)\{ ({\sigma_0}-1)[\log(t+e)+1] + {\sigma_0}(m-1) \}}} \left[ \log(t+e) \, C_\sigma^2 \right]^{\frac{\sigma}{\sigma m-1}} t^{-\frac{\sigma}{\sigma m-1}} \qquad \forall t>0 \, ,
\end{equation*}
which can be rewritten as (recall \eqref{eq:log-gamma-xx})
\begin{equation}\label{eq:log-gamma-9}
\left\| u(\cdot,t) \right\|_{\infty} \leq  K \, (\sigma-1)^{\frac{2\beta \sigma({\sigma_0}m-1)}{(2-\beta)(\sigma m-1)\{ ({\sigma_0}-1)[\log(t+e)+1] + {\sigma_0}(m-1) \}}} \left[ \log(t+e) \left(\sigma-1\right)^{-\frac{2\beta}{2-\beta}} \right]^{\frac{\sigma}{\sigma m-1}} t^{-\frac{\sigma}{\sigma m-1}} 
\end{equation}
for all $t>0$. We can now let also $ \sigma $ vary upon setting  
$$ \sigma = 1 + \frac{{\sigma_0}-1}{\log(t+e)} \, , $$
so that \eqref{eq:log-gamma-9} (using the boundedness of the first factor) yields
\begin{equation*}\label{eq:log-gamma-10}
\begin{aligned}
\left\| u(\cdot ,t) \right\|_{\infty} \leq & \, K \left[ \log(t+e) \right]^{\frac{2+\beta}{(m-1)(2-\beta)} - \frac{(2+\beta)({\sigma_0}-1)}{[(m-1)\log(t+e)+m({\sigma_0}-1) ](m-1)(2-\beta)} } \\
& \times   t^{-\frac{1}{m-1} + \frac{{\sigma_0}-1}{[(m-1)\log(t+e)+m({\sigma_0}-1) ](m-1)}} \qquad \forall t>0 \, ,
\end{aligned}
\end{equation*}
namely
\begin{equation*}\label{eq:log-gamma-11}
\left\| u(\cdot,t) \right\|_{\infty} \leq  K \left[ \log(t+e) \right]^{\frac{2+\beta}{(m-1)(2-\beta)}} t^{-\frac{1}{m-1} } \qquad \forall t>0 \, ,
\end{equation*}
which is precisely \eqref{eq: log-alfa} in the case $ \| u_0 \|_1 = 1 $. As for optimality, it is enough to invoke \cite[Theorem 3.2]{GMV}: if the curvature assumption \eqref{sezionale-smooth} holds with reverse inequality, then all nontrivial, bounded, compactly supported and nonnegative initial datum $ u_0 $ gives rise to a solution of \eqref{PME} satisfying (in particular) the lower bound
\begin{equation}\label{esB}
\left\| u(\cdot,t) \right\|_\infty^{m-1} \ge \hat{K} \, \frac{\left( \log t \right)^{\frac{2+\beta}{2-\beta}}}{t} \qquad \text{for large } t \, ,
\end{equation}
where $ \hat{K} $ is a suitable positive constant depending only on $ m,N,\psi, u_0 $. It is plain that \eqref{esB} matches \eqref{eq: log-alfa} from below with respect to long-time behavior, up to constants.
\end{proof}

\begin{remark}\rm
We point out that Theorem \ref{thm: sobo-deg}  holds for $ \beta=0 $ as well: in fact in such case the result is true for all $ L^1(\mathbb{M}^N) $ initial data, not only the radial ones, and $ \mathbb{M}^N $ need not be a model manifold. This is a direct consequence of Theorem \ref{mckball} below and \cite[Theorem 2.1]{GM16}, whereas optimality is due to the sharp estimates of \cite{VazHyp}. 
\end{remark}

In the quasi-Euclidean case, thanks to Theorem \ref{teo2} we can obtain the analogue of Theorem \ref{thm: sobo-deg}. The proof, that we omit, follows by combining \cite[Corollary 5.6]{GM13} (the fact that it is stated on Euclidean domains is irrelevant) and \cite[Theorem 6.2]{GMV}.
\begin{theorem}\label{thm: sobo-deg-quasi}
Let $ \mathbb{M}^N_\psi $ be a Cartan-Hadamard model manifold satisfying 
\begin{equation}\label{sezionale2-quasi} 
\mathrm{Sect}_\omega(x)  = \frac{1}{N-1} \, \mathrm{Ric}_o(x) = - \frac{\psi^{\prime\prime}(r)}{\psi(r)} \leq - C_1 \, r^{-2} \qquad \forall x \in \mathbb{M}^N_\psi \setminus B_{R_0} \, ,
\end{equation}
for some $ C_1 , R_0>0$. Then there exists $ K>0 $, depending only on $ N , C_1 , R_0$, such that for all $ u_0 \in L^1_{\mathrm{rad}}\big(\mathbb{M}^N_\psi\big) $ the solution $u$ of \eqref{PME} fulfills the smoothing estimate
\begin{equation}\label{eq: log-alfa-quasi}
\left\| u(\cdot,t) \right\|_{L^\infty(\mathbb{M}^N_\psi)} \le K \, t^{-\frac{\tilde{N}}{2 + \tilde{N}^{{\phantom{a}}^{\phantom{a}}}\!\!\!\!\!\!\!(m-1)}} \left\| u_0 \right\|_{L^1(\mathbb{M}^N_\psi)}^{\frac{2}{2 + \tilde{N}^{{\phantom{a}}^{\phantom{a}}}\!\!\!\!\!\!\!(m-1)}}  \qquad \forall t>0 \, ,
\end{equation}
where $ \tilde{N} $ is defined in \eqref{tsfinaleX}. The result is optimal w.r.t.~long-time behavior, in the sense that if \eqref{sezionale2-quasi} holds with reverse inequality then there exist initial data $ u_0 \in L^1_{\mathrm{rad}}\big(\mathbb{M}^N_\psi\big) $ for which \eqref{eq: log-alfa-quasi} holds with reverse inequality for large $t$.
\end{theorem}

\appendix
\section{The Poincar\'e inequality}\label{sec-mckean}

As discussed in the Introduction, one of our main motivations was a remarkable paper by H.P. McKean, which is fully devoted to the proof of the following result. 

\begin{theorem}[{\cite[Statement on page 360]{McKean}}]\label{mck-orig} 
	Consider a smooth, $N$-dimensional, simply-connected Riemannian manifold $ {M} $ with negative sectional curvatures $ \mathrm{Sect} $ bounded away from $0$: specifically, suppose $ \mathrm{Sect} \le -k $ for some constant $k > 0$. Then the spectrum of the corresponding Laplace-Beltrami operator $\Delta$ acting in $ L^2 ({M}) $ is also bounded from $0$: specifically, the top of the spectrum lies to the left of
	\begin{equation*}
	-\dfrac{k \left( N-1 \right)^2 }{4} \, ,
	\end{equation*}
	and this bound is sharp.
\end{theorem}

Note that Theorem \ref{mck-orig} can be rephrased equivalently in this way: on any Cartan-Hadamard manifold $ \mathbb{M}^N $ with sectional curvatures bounded from above by $ -k<0 $, the \emph{Poincar\'e inequality} 
\begin{equation}\label{pp} 
\left\| f \right\|_{L^2(\mathbb{M}^N)} \le \frac{2}{\sqrt{k} \left( N-1 \right) } \left\| \nabla f \right\|_{L^2(\mathbb{M}^N)} \qquad \forall f \in C^1_c \!\left(\mathbb{M}^N\right)  
\end{equation}
holds. This is the form of the statement that we will refer to in the sequel.

\vspace{0.1cm}

The original proof of McKean is far from trivial. In contrast to what we have shown above, in the pure Poincar\'e case ($ p=2 $) it is enough to establish the inequality for \emph{radial} functions, since the extension from radial to nonradial is straightforward, see Theorem \ref{thm:mckean} below. In order to prove that the weight associated with the volume measure on $ \mathbb{M}^N $ (recall Subsection \ref{nfrg}) satisfies a differential inequality of the type of \eqref{hp-psi} (actually of second order) w.r.t.~the variable $r$, which is at the core of the problem, the author used several techniques that involve the second fundamental form, Jacobi fields and the so-called index form of Morse theory. Here we will only employ arguments from weighted one-dimensional inequalities, in the spirit of Section \ref{sec:radial}. The main nontrivial tools behind our methods are the Laplacian-comparison theorems recalled in Subsection \ref{lc}, which allow one to pass from model manifolds to general manifolds with little effort. Furthermore, we are able to slightly extend the validity of McKean's Theorem, as we require that only the \emph{radial} sectional curvatures are bounded away from zero.

\vspace{0.1cm}

Before carrying out our alternative proof, we need a preliminary result.  

\begin{lemma}\label{lem-mck}
	Let $ \psi \in \mathcal{A} $. Let $ N \in \mathbb{N} $ with $N \ge 2$. If 
	\begin{equation}\label{hp-psi}
	\frac{\psi^\prime(r)}{\psi(r)} \ge \sqrt{k} \qquad \forall r>0 \, , \quad \text{for some } \, k>0 \, ,
	\end{equation}
	then
	\begin{equation}\label{hp-psi-cons}
	\sup_{r\in(0,\infty)} \left(\int_0^r\psi(s)^{N-1} \, ds\right)^{\frac{1}{2}}\left(\int_r^\infty \frac{1}{\psi(s)^{N-1}} \, ds \right)^{\frac 1 2}\leq \dfrac{1}{\sqrt{k}\left( N-1 \right) }  \, . 
	\end{equation}
\end{lemma}
\begin{proof}
	For convenience, let us assume in the first place that $ k=1 $: the general case will be briefly discussed at the end of the proof. Given $ \varepsilon>0 $, the integration of \eqref{hp-psi} from $\varepsilon$ to $r$ yields
	\begin{equation}\label{eq:stima-psi-below}
	\psi(r) \ge \psi(\varepsilon) \, e^{r-\varepsilon} \qquad \forall r \ge \varepsilon \, .
	\end{equation}
	On the other hand, inequality \eqref{hp-psi} can be rewritten as  
	\begin{equation*}\label{disfond McKean}
	\psi(r)^{N-1} \le \frac{1}{N-1} \, \dfrac{d}{dr} \left( \psi^{N-1} \right) \! (r) \qquad \forall r > 0 \, ,
	\end{equation*}
	so that an integration between $0$ and $ r $ yields (recall that $ \psi(0)=0 $) 
	\begin{equation}\label{est-psi-parts McKean}
	\int_{0}^r \psi(s)^{N-1} \, ds \le \dfrac{1}{N-1} \, \psi(r)^{N-1} \qquad \forall r > 0 \, .
	\end{equation}
	Similarly we obtain 
	\begin{equation}\label{est-psi-parts McKean-bis}
	\int_{r}^\infty \frac{1}{\psi(s)^{N-1}} \, ds \le \dfrac{1}{N-1} \, \frac{1}{\psi(r)^{N-1}} \qquad \forall r > 0 \, ,
	\end{equation}
	where we have exploited the fact that $ \lim_{r \to \infty} \psi(r) = \infty $, trivial consequence of \eqref{eq:stima-psi-below}. By combining \eqref{est-psi-parts McKean} and \eqref{est-psi-parts McKean-bis}, we finally deduce that
	$$
	\int_{0}^r \psi(s)^{N-1} \, ds \int_{r}^\infty \frac{1}{\psi(s)^{N-1}} \, ds \le \frac{1}{(N-1)^2} \qquad \forall r >0 \, ,
	$$
	i.e.~\eqref{hp-psi-cons} for $ k=1 $. The general case follows by applying the argument to $ r \mapsto  \sqrt{k} \, \psi\big( {r}/{\sqrt{k}} \big) $. 
\end{proof}

We are now ready to give a direct proof of McKean's Theorem.

\begin{theorem}\label{thm:mckean}
	Let $ \mathbb{M}^N $ be a Cartan-Hadamard manifold such that
	\begin{equation}\label{hp-sect} 
	\mathrm{Sect}_\omega(x) \le -k \qquad \forall x \in \mathbb{M}^N \, ,
	\end{equation}
	for some $ k>0 $. Then the Poincar\'e inequality \eqref{pp} holds.
\end{theorem}
\begin{proof}
	Thanks to \eqref{hp-sect}, we can apply the Laplacian-comparison results recalled in Subsection \ref{lc} (in particular \eqref{comp-sect-2}) to the explicit model function $ \psi(r) = \sinh\!\big(\sqrt{k} r \big)/\sqrt{k} $, which corresponds to the hyperbolic space of curvature $-k$ and satisfies
	\begin{equation*}\label{hyp-k}
	\frac{\psi^{\prime\prime}(r)}{\psi(r)} = k \qquad \forall r>0 \, .
	\end{equation*} 
	Upon recalling identity \eqref{lap-m}, we have:
	\begin{equation*}\label{c1}
	\begin{gathered}
	\frac{\frac{\partial}{\partial r} A(r,\theta) }{A(r,\theta)} = \mathsf{m}(r,\theta)  \ge (N-1) \, \frac{\psi^\prime(r)}{\psi(r)} = (N-1)\, \sqrt{k} \, \coth\!\big( \sqrt{k} r \big) \ge (N-1) \, \sqrt{k} \\
	 \forall (r,\theta) \in \mathbb{R}^+ \times \left( \mathbb{S}^{N-1} \setminus \mathcal{P} \right) ,
	\end{gathered}
	\end{equation*}  
	namely
	\begin{equation}\label{c2}
	\frac{\frac{\partial}{\partial r} {\psi_A}(r,\theta) }{{\psi_A}(r,\theta)} \ge  \sqrt{k} \quad \forall (r,\theta) \in \mathbb{R}^+ \times \left( \mathbb{S}^{N-1} \setminus \mathcal{P} \right) , \qquad \text{where } \, {\psi_A}(r,\theta) := A(r,\theta)^{\frac{1}{N-1}} \, .
	\end{equation}
	By reasoning similarly to the proof of Theorem \ref{teo}, it is not difficult to check that $ \psi_A(\cdot,\theta) \in \mathcal{A} $ for all $ \theta \in \mathbb{S}^{N-1} \setminus \mathcal{P} $. Hence in view of \eqref{c2} we can exploit Lemma \ref{lem-mck}, which ensures that 
	\begin{equation*}\label{hp-psi-cons-A}
	\sup_{r\in(0,\infty)} \left(\int_0^r \psi_A(s,\theta)^{N-1} \, ds\right)^{\frac{1}{2}}\left(\int_r^\infty \frac{1}{\psi_A(s,\theta)^{N-1}} \, ds \right)^{\frac 1 2}\leq \dfrac{1}{\sqrt{k}\left( N-1 \right) } \qquad \forall \theta \in \mathbb{S}^{N-1}  \setminus \mathcal{P} \, . 
	\end{equation*}
	As a consequence, from Proposition \ref{thm:ko} with $ p=2 $ and $ w(r) = \psi_A(r,\theta) $ we deduce that
	\begin{equation}\label{hp-psi-cons-B}
	\int_0^\infty \left|g(r)\right|^2 A(r,\theta) \, dr \le \dfrac{4}{{k}\left( N-1 \right)^2} \, \int_0^\infty \left|g^\prime(r)\right|^2  A(r,\theta) \, dr \quad \, \forall  g \in C^1_c([0,\infty)) \, , \quad \forall \theta \in  \mathbb{S}^{N-1} \setminus \mathcal{P} \, .
	\end{equation}
	On the other hand, if $ f \in C^1_c(\mathbb{M}^N) $ then $ r \mapsto f(r,\theta) \in C^1_c([0,\infty)) $ for every $ \theta \in \mathbb{S}^{N-1} $, so that by using \eqref{hp-psi-cons-B} with $ g(r) = f(r,\theta) $ and integrating over $ \mathbb{S}^{N-1} $ we end up with
	$$
	\begin{aligned}
	\int_{\mathbb{S}^{N-1}} \int_0^\infty \left|f(r,\theta)\right|^2 A(r,\theta) \, dr \, d \theta \le & \, \dfrac{4}{{k}\left( N-1 \right)^2} \, \int_{\mathbb{S}^{N-1}} \int_0^\infty \left| \frac{\partial}{\partial r} {f}(r,\theta) \right|^2 A(r,\theta) \, dr \, d \theta \\
	\le & \, \dfrac{4}{{k}\left( N-1 \right)^2} \, \int_{\mathbb{S}^{N-1}} \int_0^\infty \left| \nabla {f}(r,\theta) \right|^2 A(r,\theta) \, dr \, d \theta \, ,
	\end{aligned}
	$$
	namely \eqref{pp} thanks to \eqref{LP} and Fubini's Theorem. 
\end{proof}

By combining the techniques of proof of Theorems \ref{teo} and \ref{thm:mckean}, one can deduce the validity of the Poincar\'e inequality (with a generic constant) upon only requiring that the radial sectional curvatures are bounded above by  $ -k $ in the \emph{complement of a ball}. Because the argument does not deviate significantly from the previous ones, we omit the proof and refer to \cite[Subsection 4.4.1]{R} for the details.

\begin{theorem}\label{mckball}
	Let $\mathbb{M}^N$ be a Cartan-Hadamard manifold such that
	\begin{equation*}\label{hp-sect_McKean bis} 
	\mathrm{Sect}_\omega(x) \le -k \qquad \forall x \in \mathbb{M}^N\setminus B_{R_0} \,,
	\end{equation*}
	for some $ k,R_0>0$. Then there exists a positive constant $C$, depending only on $k,R_0$, such that
	\begin{equation*}\label{poin-mckbis}
	\left\| f \right\|_{L^2(\mathbb{M}^N)} \le C \left\| \nabla f \right\|_{L^2(\mathbb{M}^N)} \qquad \forall f \in C^1_c \!\left(\mathbb{M}^N\right) .
	\end{equation*}
\end{theorem}

\medskip 

\noindent{\textbf{Acknowledgements.}} The authors thank the ``Gruppo Nazionale per l'Analisi Matematica, la Probabilit\`a e le loro Applicazioni'' (GNAMPA) of the ``Istituto Nazionale di Alta Matematica'' (INdAM, Italy). M.M.~was partially supported by the GNAMPA  Project 2017 ``Equazioni Diffusive Non-lineari in Contesti Non-Euclidei e Disuguaglianze Funzionali Associate'' and both authors were supported by the GNAMPA Project 2018 ``Problemi Analitici e Geometrici Associati a EDP Non-Lineari Ellittiche e Paraboliche''. M.M.~also thanks Prof.~Jean Dolbeault for some inspiring discussions regarding Section \ref{sect:nonrad}.


\begin{thebibliography}{50} 

\bibitem{Aubin1} T. Aubin, \emph{\'Equations diff\'erentielles non lin\'eaires et probl\`eme de Yamabe concernant la courbure scalaire}, J. Math. Pures Appl. \textbf{55} (1976), 269--296.

\bibitem{Aubin3} T. Aubin, \emph{Espaces de Sobolev sur les vari\'et\'es Riemanniennes}, Bull. Sci. Math. \textbf{100} (1976), 149--173.

\bibitem{Aubin2} T. Aubin, \emph{Probl\`emes isop\'erim\'etriques et espaces de Sobolev}, J. Differential Geometry \textbf{11} (1976), 573--598.

\bibitem{Aubin Li} T. Aubin and Y.Y. Li, \emph{On the best Sobolev inequality}, J. Math. Pures Appl. \textbf{78} (1999), 353--387.

\bibitem{BCLS} D. Bakry, T. Coulhon, M. Ledoux and L. Saloff-Coste, \emph{Sobolev inequalities in disguise}, Indiana Univ. Math. J. \textbf{44} (1995), 1033--1074.

\bibitem{BGL} D. Bakry, I. Gentil and M. Ledoux, ``Analysis and Geometry of Markov Diffusion Operators'', Grundlehren der Mathematischen Wissenschaften, 348. Springer, Cham, 2014.

\bibitem{BLW} D. Bakry, M. Ledoux and F.-Y. Wang, \emph{Perturbations of functional inequalities using growth conditions}, J. Math. Pures Appl. \textbf{87} (2007), 394--407.

\bibitem{BGG} E. Berchio, D. Ganguly and G. Grillo, \emph{Sharp Poincar\'e-Hardy and Poincar\'e-Rellich inequalities on the hyperbolic space}, J. Funct. Anal. \textbf{272} (2017), 1661--1703. 

\bibitem{BG05} M. Bonforte and G. Grillo, \emph{Asymptotics of the porous media equation via Sobolev inequalities}, J. Funct. Anal. \textbf{225} (2005), 33--62.

\bibitem{BGV08} M. Bonforte, G. Grillo and J.L. V\'azquez, \emph{Fast diffusion flow on manifolds of nonpositive curvature}, J. Evol. Equ. \textbf{8} (2008), 99--128.

\bibitem{Car1} G. Carron, \emph{In\'egalit\'es de Hardy sur les vari\'et\'es riemanniennes non-compactes}, J. Math. Pures Appl. \textbf{76} (1997), 883--891. 

\bibitem{Car2} G. Carron, \emph{In\'egalit\'e de Sobolev et volume asymptotique}, Ann. Fac. Sci. Toulouse Math. \textbf{21} (2012), 151--172.

\bibitem{Dav} E.B. Davies, ``Heat Kernels and Spectral Theory'', Cambridge Tracts in Mathematics, 92. Cambridge University Press, Cambridge, 1990.

\bibitem{DEL} J. Dolbeault, M.J. Esteban and M. Loss, \emph{Rigidity versus symmetry breaking via nonlinear flows on cylinders and Euclidean spaces}, Invent. Math. \textbf{206} (2016), 397--440.

\bibitem{DELM} J. Dolbeault, M.J. Esteban, M. Loss and M. Muratori, \emph{Symmetry for extremal functions in subcritical Caffarelli-Kohn-Nirenberg inequalities}, C. R. Math. Acad. Sci. Paris \textbf{355} (2017), 133--154.

\bibitem{FM} A.R. Fotache and M. Muratori, \emph{Smoothing effects for the filtration equation with different powers}, J. Differential Equations \textbf{263} (2017), 3291--3326. 

\bibitem{GS} M. Ghomi and J. Spruck, \emph{Total curvature and the isoperimetric inequality in Cartan-Hadamard manifolds}, preprint arXiv: \url{https://arxiv.org/abs/1908.09814}. 

\bibitem{GreeneWu} R.E. Greene and H. Wu, ``Function Theory on Manifolds which Possess a Pole'', Lecture Notes in Mathematics, 699, Springer, Berlin, 1979.

\bibitem{Grigor'yan} A. Grigor'yan, \emph{Analytic and geometric background of recurrence and non-explosion of the Brownian motion on Riemannian manifolds}, Bull. Amer. Math. Soc. \textbf{36} (1999), 135--249.

\bibitem{Grig09} A. Grigor'yan, ``Heat Kernel and Analysis on Manifolds'', AMS/IP Studies in Advanced Mathematics, 47, American Mathematical Society, Providence, RI; International Press, Boston, MA, 2009.

\bibitem{GSC} A. Grigor'yan and L. Saloff-Coste, \emph{Surgery of the Faber-Krahn inequality and applications to heat kernel bounds}, Nonlinear Anal. \textbf{131} (2016), 243--272.

\bibitem{GM13} G. Grillo and M. Muratori, \emph{Sharp short and long time $L^\infty$ bounds for solutions to porous media equations with homogeneous Neumann boundary conditions}, J. Differential Equations \textbf{254} (2013), 2261--2288. 

\bibitem{GM16} G. Grillo and M. Muratori, \emph{Smoothing effects for the porous medium equation on Cartan-Hadamard manifolds}, Nonlinear Anal. \textbf{131} (2016), 346--362.

\bibitem{GMP13} G. Grillo, M. Muratori and M.M. Porzio, \emph{Porous media equations with two weights: smoothing and decay properties of energy solutions via Poincar\'e inequalities}, Discrete Contin. Dyn. Syst. \textbf{33} (2013), 3599--3640.

\bibitem{GMV} G. Grillo, M. Muratori and J.L. V\'azquez, \emph{The porous medium equation on Riemannian manifolds with negative curvature. The large-time behaviour}, Adv. Math. \textbf{314} (2017), 328--377. 

\bibitem{Hebey ter} E. Hebey, \emph{Optimal Sobolev inequalities on complete Riemannian manifolds with Ricci curvature bounded below and positive injectivity radius}, Amer. J. Math. \textbf{118} (1996), 291--300.

\bibitem{Hebey bis} E. Hebey, ``Sobolev Spaces on Riemannian Manifolds'', Lecture Notes in Mathematics, 1635. Springer-Verlag, Berlin, 1996. 

\bibitem{Hebey} E. Hebey, ``Nonlinear Analysis on Manifolds: Sobolev Spaces and Inequalities'', Courant Lecture Notes in Mathematics, 5. New York University, Courant Institute of Mathematical Sciences, New York; American Mathematical Society, Providence, RI, 1999.  

\bibitem{HV} E. Hebey and M. Vaugon, \emph{The best constant problem in the Sobolev embedding theorem for complete Riemannian manifolds}, Duke Math. J. \textbf{79} (1995), 235--279.

\bibitem{Kufner} A. Kufner and B. Opic, ``Hardy-type Inequalities'', Pitman Research Notes in Mathematics Series, 219. Longman Scientific \& Technical, Harlow, 1990.

\bibitem{L99} M. Ledoux, \emph{On manifolds with non-negative Ricci curvature and Sobolev inequalities}, Comm. Anal. Geom. \textbf{7} (1999), 347--353. 
 
\bibitem{L04} M. Ledoux, \emph{Spectral gap, logarithmic Sobolev constant, and geometric bounds}, In: ``Eigenvalues of Laplacians and other Geometric Operators'', A. Grigor’yan and S.-T. Yau (eds.), Surveys in Differential Geometry. Vol. IX, Int. Press, Somerville, MA, 2004, 219--240.

\bibitem{LZ} Z. Lu and D. Zhou, \emph{On the essential spectrum of complete non-compact manifolds}, J. Funct. Anal. \textbf{260} (2011), 3283--3298.

\bibitem{McKean} H.P. McKean, \emph{An upper bound to the spectrum of $\Delta$ on a manifold of negative curvature}, J. Differential Geometry \textbf{4} (1970), 359--366.

\bibitem{Nguyen} V.H. Nguyen, \emph{The sharp Poincar\'e-Sobolev type inequalities in the hyperbolic space $\mathbb{H}^n$}, J. Math. Anal. Appl. \textbf{462} (2018), 1570--1584. 

\bibitem{PiVe} S. Pigola and G. Veronelli, \emph{Lower volume estimates and Sobolev inequalities}, Proc. Amer. Math. Soc. \textbf{138} (2010), 4479--4486.

\bibitem{R} A. Roncoroni, ``Symmetry and Quantitative Stability Results for Problems in Geometric Analysis and Functional Inequalities'', Ph.D.~Thesis, Universit\`a degli Studi di Pavia, \url{https://iris.unipv.it/retrieve/handle/11571/1292128/301088/Tesi.pdf}.

\bibitem{Talenti} G. Talenti, \emph{Best constant in Sobolev inequality}, Ann. Mat. Pura Appl. \textbf{110} (1976), 353--372.

\bibitem{V1} J.L. V\'azquez, ``Smoothing and Decay Estimates for Nonlinear Diffusion Equations. Equations of Porous Medium Type'', Oxford
University Press, Oxford, 2006.

\bibitem{V2} J.L. V\'azquez, ``The Porous Medium Equation. Mathematical Theory'', The Clarendon Press, Oxford University Press, Oxford,
2007.

\bibitem{VazHyp} J.L. V\'azquez, \emph{Fundamental solution and long time behavior of the porous medium equation in hyperbolic space}, J. Math. Pures Appl. \textbf{104} (2015), 454--484.

\bibitem{W} J.P. Wang, \emph{The spectrum of the Laplacian on a manifold of nonnegative Ricci curvature}, Math. Res. Lett. \textbf{4} (1997), 473--479.

\bibitem{X} Y.L. Xin, ``Geometry of Harmonic Maps'', Progress in Nonlinear Differential Equations and their Applications, 23. Birkh\"{a}user Boston, Inc., Boston, MA, 1996.

\end{thebibliography}
\end{document}